\documentclass[leqno,11pt]{article}
\usepackage{amssymb,amsmath,amsfonts,amsthm}
\usepackage{authblk,  empheq}
\usepackage{epsfig,color,mathrsfs}
\usepackage{graphicx}
\usepackage{amscd}
\usepackage{caption}
\usepackage{mathrsfs,bbm}
\usepackage{multicol}
\usepackage{color}

\newcommand{\be}{\begin{equation}}
\newcommand{\ee}{\end{equation}}
\newcommand{\ben}{\begin{eqnarray*}}
\newcommand{\een}{\end{eqnarray*}}

\newcommand{\ds}{\displaystyle}

\newcommand{\R}{\mathbb R}

\allowdisplaybreaks

\newtheorem{theorem}{Theorem}[section]

\newtheorem{lemma}[theorem]{Lemma}

\newtheorem{proposition}[theorem]{Proposition}
\newtheorem{remark}[theorem]{Remark}

\definecolor{darkgreen}{rgb}{0.09, 0.45, 0.27}
\definecolor{debianred}{rgb}{0.84, 0.04, 0.33}

\numberwithin{equation}{section}
\begin{document}
\title{\vskip-0.3in Steady-states of the Gierer-Meinhardt system in exterior domains} 

\author[ ]{Marius Ghergu$^{1,2}$}

\author[ ]{Jack McNicholl$\,^{1}$}

\affil[ ]{$^1$School of Mathematics and Statistics}
\affil[ ]{University College Dublin }
\affil[ ]{Belfield Campus, Dublin 4, Ireland}
\affil[ ]{E-mail: {\tt marius.ghergu@ucd.ie}}
\affil[ ]{E-mail: {\tt jack.mcnicholl@ucdconnect.ie}}
\affil[ ]{}
\affil[ ]{$^2$Institute of Mathematics Simion Stoilow of the Romanian Academy}
\affil[ ]{21 Calea Grivitei St., 010702 Bucharest, Romania}


\maketitle

\begin{abstract} 
We discuss the existence and nonexistence of solutions to the steady-state Gierer-Meinhardt system
$$
\begin{cases}
\displaystyle   -\Delta u=\frac{u^p}{v^q}+\lambda \rho(x) \,, u>0 &\quad\mbox{ in }\mathbb{R}^N\setminus K,\\[0.1in]
\displaystyle   -\Delta v=\frac{u^m}{v^s}  \,, v>0 &\quad\mbox{ in }\mathbb{R}^N\setminus K,\\[0.1in]
\displaystyle \;\;\; \frac{\partial u}{\partial \nu}=\frac{\partial v}{\partial \nu}=0 &\quad\mbox{ on }\partial K,\\[0.1in]
\displaystyle \;\;\;  u(x), v(x)\to 0 &\quad\mbox{ as }|x|\to \infty,
\end{cases}
$$
where $K\subset \mathbb{R}^N$ $(N\geq 2)$ is a compact set, $\rho\in C^{0,\gamma}_{loc}(\overline{\mathbb{R}^N\setminus K})$, $\gamma\in (0,1)$, is a nonnegative function  and  $p,q,m,s, \lambda>0$.  Combining fixed point arguments with suitable barrier functions, we construct solutions with a prescribed asymptotic growth at infinity.  Our approach can be extended to many other classes of semilinear elliptic systems with various sign of exponents. 
\end{abstract}


\medskip



\section{Introduction}


Alan Turing (1912 –1954) was a British scientist credited for developing the first modern computer, artificial intelligence and for breaking the Enigma code during World War II. One of his scientific legacies pertains to the chemical basis of morphogenesis, that is, a biological process that causes a tissue or organ to develop its shape through the spatial distribution of its cells.
Turing's research \cite{T52} was related to  pattern formation due to the breakdown of symmetry and homogeneity in
initially homogeneous continuous media. He  concluded that small variances in chemical concentrations
can amplify and act as symmetry-breaking in the biological development process and thus, they may give rise to patterns. Mathematically, this process is described by two reaction-diffusion equations where the unknowns quantify the  concentration of two chemicals.  Let us point out that patterns may result from different processes such as cell polarity (Lawrence \cite{L66}) or the presence of two periodic
events featuring different wavelengths (Goodwin and Cohen \cite{GC69}).

Based on Turing’s idea,  Gierer and Meinhardt \cite{GM72} proposed in 1972 a mathematical model for pattern formation that emphasizes the existence of two antagonist chemicals: there is a short-range autocatalytic activator on one hand, and a long-range, cross-catalytic inhibitor on the other hand, see Figure 1. 

\begin{figure}[h!]
\begin{center}
\includegraphics[width=11cm]{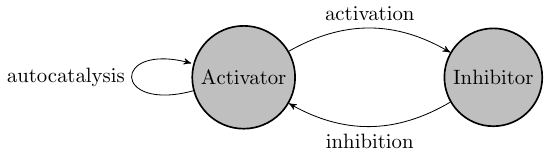}
\caption{The activator-inhibitor interdependence in the Gierer-Meinhardt model.}
\end{center}
\label{ai}
\end{figure}

Their mathematical model, also expressed as a reaction-diffusion system, reads as follows:
$$
\begin{cases}
\ds a_t-d_1\Delta a=-\mu_a a+\rho \frac{a^p}{h^q}+\rho_a,\\[0.4cm]
\ds h_t-d_2\Delta h=-\mu_h h +\rho \frac{a^m}{h^s},\\[0.2cm]
\end{cases}
$$
where:
\begin{itemize}
\item $a$ and $h$ represent the concentration of the activator and inhibitor respectively.
\item $a_t$ and $b_t$ account for the change of activator and inhibitor in time unit.
\item $\rho$ is the self-production rate of the activator.
\item the nonlinear terms $\ds \frac{a^p}{h^q}$ and $\ds \frac{a^m}{h^s}$ describe the production rates of the activator and inhibitor in the process of tissue formation. They depend on the nonlinear terms $a^p$ and $a^m$, being slowed down by the inhibitor input $\ds 1/h^q$ and $\ds 1/h^s$ respectively. It is assumed that
$p>1$ and $q,m, s>0$.

The above system is said to have a common source for activator and inhibitor if $(p,q)=(m,s)$.
\item $\mu_a$, $\mu_h$ represent the decay rates of the two chemicals.
\item $\rho_a$ is a small activator-independent production rate of the activator and is required to initiate the activator autocatalysis at very low activator concentration.

\item $d_1$ and $d_2$ are the diffusion rates of the activator and inhibitor respectively. The activator diffuses slowly while the inhibitor has a fast diffusion, so $d_1<<d_2$.

\end{itemize}
From mathematical point of view, the Gierer-Meinhardt system has been widely investigated in many settings and contexts \cite{DKZ21, GT14, G09, H24,  I20, J06, KBM21, KS17, KPW20, KPW21, K10, H15}. More recent results are related to the Gierer-Meinhardt system in unbounded domains \cite{DKW03, G23, G24, GDW24,  KWY13}. For instance, 
Kolokolnikov, Wei and  Yang \cite{KWY13} discussed the existence of positive solutions in $\R^3$ to the following  system:
\begin{equation}\label{WW1}
\begin{cases}
\displaystyle   -\varepsilon \Delta u+ u=\frac{u^2}{v}&\quad\mbox{ in }\R^3,\\[0.1in]
\displaystyle   -\Delta v+ v=u^2 &\quad\mbox{ in }\R^3,\\[0.1in]
u(x), v(x)\to 0& \mbox{ as }|x|\to \infty.
\end{cases}
\end{equation}
One of the main results of \cite{KWY13} is the existence of positive solutions to \eqref{WW1} for $\varepsilon>0$ sufficiently small. Such solutions are radial and  decay exponentially at infinity. Similar results in the plane are obtained  in \cite{DKW03}.
The following system related to \eqref{WW1} was studied in \cite{G23}:
$$
\begin{cases}
\displaystyle   -\Delta u+ \lambda u=\frac{u^p}{v^q}+\rho(x)&\quad\mbox{ in }\R^N, N\geq 3,\\[0.1in]
\displaystyle   -\Delta v+ \mu v=\frac{u^m}{v^s} &\quad\mbox{ in }\R^N,\\[0.1in]
u(x), v(x)\to 0& \mbox{ as }|x|\to \infty,
\end{cases}
$$
where $\rho\in C(\R^N)$, $\lambda, \mu\geq 0$ and $p,q,m,s>0$. Various type of solutions with either exponential or power type decay were obtained in \cite{G23}. This study was further continued in \cite{G24} where the Gierer-Meinhardt system is considered in the upper half-space. The approach in \cite{G23, G24} relies heavily on the integral representation formulae of solutions to some semilinear PDEs in $\R^N$ and on various integral estimates. 

In the present work we use a different approach to investigate qualitative properties of solutions to the steady-state Gierer-Meinhardt system in exterior domains. Precisely, we are concerned with the existence and nonexistence of positive solutions to the system 
\begin{equation}\label{GM0}
\begin{cases}
\displaystyle   -\Delta u=\frac{u^p}{v^q}+\lambda \rho(x) \,, u>0 &\quad\mbox{ in }\R^N\setminus K \, (N\geq 2),\\[0.1in]
\displaystyle   -\Delta v=\frac{u^m}{v^s}  \,, v>0 &\quad\mbox{ in }\R^N\setminus K,\\[0.1in]
\displaystyle \;\;\; \frac{\partial u}{\partial \nu}=\frac{\partial v}{\partial \nu}=0 &\quad\mbox{ on }\partial K,\\[0.1in]
\displaystyle \;\;\;  u(x), v(x)\to 0 &\quad\mbox{ as }|x|\to \infty,
\end{cases}
\end{equation}
where $\rho\in C^{0,\gamma}_{loc}(\overline{\R^N\setminus K})$ is a nonnegative function, $\gamma\in (0,1)$ and $p,q,m,s, \lambda>0$.  We are looking for positive classical solutions of \eqref{GM0}, that is, functions $u, v\in C^2(\overline{\R^N\setminus K})$ such that $u, v>0$ and satisfy \eqref{GM0} at every point of $\overline{\R^N\setminus K}$. 

In the above problem,  $K\subset \R^N$ $(N\geq 2)$ is a compact set that satisfies the following two conditions:

\begin{enumerate}
\item[{\rm (K1)}] $\R^N\setminus K$ is connected and there exists $r_0>0$ such that $B_{r_0}\subset K$.
\item[{\rm (K2)}] We have 
\begin{equation}\label{eqbd}
\nu(x)\cdot x\leq 0\quad\mbox{ for all }x\in \partial K,
\end{equation}
where $\nu(x)$ denotes the inner normal vector to $\partial K$ at the point $x\in \partial K$.
\end{enumerate}
As it can be easily seen, closed balls or ellipsoids whose interior  contains the origin satisfy the above condition \eqref{eqbd}. More generally, if $K$ is described geometrically by $K=\{x\in \R^N: f(x)\leq 0\}$ with a continuously differentiable function $f:\R^N\to \R$, then   \eqref{eqbd} holds provided that $\nabla f(x)\cdot x\leq 0$ for all $x\in \partial K$. 

From the biological point of view, the homogeneous Neumann conditions on $\partial K$ in \eqref{GM0} indicate that there is no exchange with the exterior of $\R^N\setminus K$. This means that the concentrations $u$ and $v$ have no effect on $K$ which is a region with a different structure, such as a hard tissue, e.g., a  bone.

\section{Main results}

Our first main result provides sufficient conditions for which \eqref{GM0} has no positive solutions. 
\begin{theorem}{\rm (Nonexistence)}\label{th1}

The system \eqref{GM0} has no positive solutions in any of the following situations:

\begin{enumerate}
\item[{\rm (i)}] $N=2$;

\item[{\rm (ii)}] $N\geq 3$ and $\ds m\leq \frac{2}{N-2}$;

\item[{\rm (iii)}] $N\geq 3$ and $\ds 0<p\leq \frac{N}{N-2}$.
\end{enumerate}
\end{theorem}
We next discuss the existence of a positive solution $(u,v)$ to \eqref{GM0} if $N\geq 3$. We have seen in Theorem \ref{th1}(iii) that \eqref{GM0} has no solutions if $N\geq 3$ and $\ds 0<p\leq \frac{N}{N-2}$. Thus, we shall assume $p> \frac{N}{N-2}$ 
and let us introduce the quantity 
$$
\sigma:=\frac{mq}{(p-1)(1+s)}>0
$$

To be more precise in our approach, we shall assume that $\rho(x)$ has the following behaviour at infinity:
\begin{equation}
\label{ro}
C_1 |x|^{-k}\leq \rho(x)\leq C_2 |x|^{-k} \quad\mbox{ in }\R^N\setminus K\quad\mbox{ for some }k>0,
\end{equation}
where $C_2>C_1>0$ are fixed constants. 

Given two positive functions $f,g\in C(\R^N\setminus K)$ we shall use the symbol $f(x)\simeq g(x)$ in $\R^N\setminus K$ to express the fact there exist $C>c>0$ such that 
$$
c\leq \frac{f(x)}{g(x)}\leq C\quad\mbox{ for all } x\in \R^N\setminus K.
$$
If $N\geq 3$ and $(u,v)$ is a positive solution of \eqref{GM0}, we have that both $u$ and $v$ are positive and superharmonic in $\R^N\setminus K$. Thus, by a known result, we have $u, v\geq c|x|^{2-N}$ in $\R^N\setminus K$, for some constant $c>0$ (see Lemma \ref{lest} below). Our first goal is to investigate the existence of positive solutions $(u,v)$ to \eqref{GM0} where the activator concentration $u(x)$ has a minimal growth at infinity, that is, 
\begin{equation}\label{beh1}
u(x)\simeq |x|^{2-N}\quad\mbox{ in }\R^N\setminus K.
\end{equation}

Our result reads as follows:
\begin{theorem}{\rm (Solutions with minimal growth in the activator concentration)}\label{th2}

Assume $N\geq 3$, $\sigma<1$ and $\rho(x)$ satisfies \eqref{ro}. If the system \eqref{GM0} has a positive solution $(u,v)$ with $u$ satisfying \eqref{beh1}, then $k>N$ and exactly one of the following conditions hold:

\begin{enumerate}
\item[{\rm (i)}]  $m\geq s+\frac{N}{N-2}$ and $p>q+\frac{N}{N-2}$;

\item[{\rm (ii)}] $m= s+\frac{N}{N-2}$, $p=q+\frac{N}{N-2}$ and $q>1+s$;

\item[{\rm (iii)}] $\frac{2}{N-2}<m< s+\frac{N}{N-2}$ and $p>\frac{q}{1+s}\Big(m-\frac{2}{N-2}\Big)+\frac{N}{N-2}$.
\end{enumerate}

Conversely, if $k>N$ and $p,q,m,s$ satisfy one of the above conditions (i)-(iii), then there exists $\lambda^*>0$ such that the system \eqref{GM0} has a positive solution $(u,v)$ for all $0<\lambda<\lambda^*$, where the activator's concentration $u(x)$ satisfies \eqref{beh1}
and the inhibitor's concentration satisfies 
\begin{equation}\label{beh2}
v(x)\simeq \begin{cases}
\ds |x|^{2-N} & \mbox{ if } m\geq s+\frac{N}{N-2} \mbox{ and } p>q+\frac{N}{N-2},\\[0.2cm]
\ds |x|^{2-N}\log^{\frac{1}{1+s}} \Big(\frac{|x|}{r_0}\Big)  & \mbox{ if } m= s+\frac{N}{N-2}\, , \; p=q+\frac{N}{N-2} \mbox{ and } q>1+s,\\[0.2cm]
\ds |x|^{-\frac{m(N-2)-2}{1+s}} & \mbox{ if }\frac{2}{N-2}<m< s+\frac{N}{N-2}\, ,\;  p>\frac{q}{1+s}\Big(m-\frac{2}{N-2}\Big)+\frac{N}{N-2}.
\end{cases}
\end{equation}
\end{theorem}
The asymptotic behaviour of $u$ and $v$ described in \eqref{beh1} and \eqref{beh2} reveals that if condition (ii) or (iii) in Theorem \ref{th2} holds, then $v(x)>u(x)$ for $|x|$ large. 

If condition (i) in Theorem \ref{th2} holds,  the asymptotics \eqref{beh1} and \eqref{beh2} yield $v(x)\simeq u(x)$ as $|x|\to \infty$. It will emerge from our proof of Theorem \ref{th2} that for $\lambda>0$ small, we have $u(x)>v(x)$ for $|x|$ large, see Remark \ref{remm}.

In the next result we construct positive solutions $(u,v)$ of \eqref{GM0} where $u(x)$ has a faster growth at  infinity than the fundamental solution $|x|^{2-N}$.

\newpage

\begin{theorem}{\rm (Solutions with faster growth  in the activator concentration)}\label{th3}

Assume $N\geq 3$, $\sigma<1$ and $\rho(x)$ satisfies \eqref{ro} with $2<k<N$. Let $a=k-2$, so that $0<a<N-2$.  Then, the system \eqref{GM0} has a positive solution $(u,v)$ with 
\begin{equation}\label{beh3}
u(x)\simeq |x|^{-a} \quad\mbox{ in }\R^N\setminus K,
\end{equation}
for some $\lambda>0$, if and only if  exactly one of the following conditions hold:

\begin{enumerate}
\item[{\rm (i)}] $m\geq \frac{N+s(N-2)}{a}$  and $p\geq q\frac{N-2}{a}+1+\frac{2}{a}$;

\item[{\rm (ii)}] $\frac{2}{a}<m< \frac{N+s(N-2)}{a}$ and $p\geq \frac{q}{1+s}\Big(m-\frac{2}{a}\Big)+1+\frac{2}{a}$.
\end{enumerate} 

Furthermore, if one of (i) or (ii) above holds and $(u,v)$ is a solution of \eqref{GM0} for some $\lambda>0$, then the inhibitor's concentration $v(x)$ satisfies
\begin{equation}\label{beh4}
v(x)\simeq \begin{cases}
\ds |x|^{2-N} & \quad \mbox{ if } m>\frac{N+s(N-2)}{a} \\[0.2cm]
\ds |x|^{2-N}\log^{\frac{1}{1+s}} \Big(\frac{|x|}{r_0}\Big)  &\quad \mbox{ if } m=\frac{N+s(N-2)}{a} \\[0.2cm]
\ds |x|^{-\frac{ma-2}{1+s}} &\quad \mbox{ if } \frac{2}{a}<m< \frac{N+s(N-2)}{a}
\end{cases}
\quad\mbox{ in }\R^N\setminus K.
\end{equation}
\end{theorem}

The remaining part of the article is organised as follows. In Section 3 we obtain various results related to the existence, nonexistence and asymptotic behaviour at infinity for the single equation 
$$
\begin{cases}
-\Delta w  = |x|^{-\alpha}  w^{-s}, w>0 & \quad\mbox{ in }\R^N\setminus K , \\[0.2cm]
\ds \;\;\; \frac{\partial w}{\partial \nu}  =0  & \quad\mbox{ on }\partial K , \\[0.2cm]
\;\;\; w(x)\to 0 &  \quad\mbox{ as }|x|\to \infty.
\end{cases}
$$
One crucial result in this sense is Proposition \ref{a2} which discusses the qualitative properties of the above problem in a more general setting. Unlike the approach in \cite{G23, G24} which relies heavily on integral representation of solutions and various integral estimates, we shall study \eqref{GM0} by constructing appropriate sub and supersolutions that fit into the geometry of the exterior domain $\R^N\setminus K$. Let us point out that integral representation of solutions which appeared in \cite{G23, G24} does not provide a convenient tool here due to the general boundary of the compact set $K$. The proofs of Theorems \ref{th1}-\ref{th3} are detailed separately in Sections 4, 5 and 6 respectively. One particular  advantage of our approach is that one can extend the study of \eqref{GM0} to include negative ranges of exponents $p,q,m,s$. This will be done in Section 7. Throughout this paper, by $c,C$ we denote generic positive constants whose values (unless otherwise stated) may change on every occurrence.

\section{Preliminaries}

In this section we collect various results needed in our approach. We start with the following known fact which provides basic estimates for positive superharmonic functions in exterior domains. 

\begin{lemma}{\rm (see \cite[Lemma 2.3]{SZ2002}) } \label{lest}
Let $\Omega\subset \R^N$ $(N\geq 2)$ be such that $\R^N\setminus B_{R}\subset \Omega$ for some $R>0$. Assume $u\in C^2(\overline \Omega)$ is positive and satisfies $-\Delta u\geq 0$ in $\Omega$. 
\begin{enumerate}
\item[{\rm (i)}] If $N=2$, then there exists $c=c(R, u)>0$ such that $u\geq c$ in $\Omega$;
\item[{\rm (ii)}] If $N\geq 3$, then there exists $C=C(N, R, u)>0$ such that $u\geq C|x|^{2-N}$ in $\Omega$.

\end{enumerate}
\end{lemma}

\begin{lemma}\label{a1}
Consider the inequality
\begin{equation}\label{ina1}
-\Delta w\geq A(|x|) g(w)\quad\mbox{ in }\R^N\setminus B_{r_0}, N\geq 2, r_0>0
\end{equation}
where $A, g\in C(0, \infty)$ are positive functions such that $g$ is convex.

Assume one of the following conditions hold:
\begin{enumerate}
\item[{\rm (i)}] $N\geq 3$, $\ds \lim_{t\to 0^+} g(t)>0$ and $\ds \int_1^\infty tA(t) dt=\infty$;  
\item[{\rm (ii)}] $N=2$, $g$ is nonincreasing\footnote{in the sense that $g(t_1)\geq g(t_2)$ for all $t_2>t_1>0$.} and for all $c>0$ we have 
$$
\ds \liminf_{t\to \infty} e^{2t} A(e^t)g(c(t+1))>0.
$$ 
\end{enumerate}
Then, \eqref{ina1} has no positive solutions $w\in C^2(\R^N\setminus B_{r_0})$.  

\end{lemma}

\begin{proof}
Suppose $w\in C^2(\R^N\setminus B_{r_0})$ is a positive solution of \eqref{ina1} and denote by $\overline  w$ its spherical average, that is,
$$
\overline w(r)=\frac{1}{|\partial B_r| }\int_{\partial B_1 } w(y) dS(y) \quad\mbox{ for all }r\geq r_0,
$$
where  $dS$ stands for the surface area measure in $\R^N$. Taking the average in \eqref{ina1} and using Jensen's inequality (since $g$ is convex) we find
\begin{equation}\label{ina2}
-\overline w''(r)-\frac{N-1}{r} \overline w'(r)\geq A(r) g(\overline w(r))\quad\mbox{ for all }r\geq r_0.
\end{equation}

(i) Using the change of variable $t=r^{2-N}$ and $W(t)=\overline w(r)$, from \eqref{ina2} we find
$$
-W''(t)\geq \frac{1}{(N-2)^2} t^{-\frac{2N-2}{N-2}} A\left(t^{-\frac{1}{N-2}}\right) g(W(t))\quad\mbox{ for all }0<t\leq t_0:=r_0^{2-N}.
$$
From the above estimate we derive that $W$ is positive and concave on $(0, t_0]$, hence bounded. Thus, since $\lim_{t\to 0^+} g(t)>0$, we derive  $g(W(t))>c>0$ for some $c>0$ and then 
$$
-W''(t)\geq C t^{-\frac{2N-2}{N-2}} A\left(t^{-\frac{1}{N-2}}\right) \quad\mbox{ for all }0<t\leq t_0.
$$
Integrating twice over $[t, t_0]$ we find
$$
\begin{aligned}
W(t_0)& \geq W(t)+W'(t_0)(t_0-t)+\int_t^{t_0}\int_\rho^{t_0} \tau^{-\frac{2N-2}{N-2}} A\left(\tau^{-\frac{1}{N-2}}\right) d\tau d\rho\\
& \geq W'(t_0)(t_0-t)+\int_t^{t_0}\int_\rho^{t_0} \tau^{-\frac{2N-2}{N-2}} A\left(\tau^{-\frac{1}{N-2}}\right) d\tau d\rho
 \quad\mbox{ for all }0<t\leq t_0.
\end{aligned}
$$
Letting $t\to 0^+$ in the above estimate we derive
$$
\infty>W(t_0)\geq t_0 W'(t_0) +\int_0^{t_0}\int_\rho^{t_0} \tau^{-\frac{2N-2}{N-2}} A\left(\tau^{-\frac{1}{N-2}}\right) d\tau d\rho.
$$
An integration by parts in the above estimate yields
$$
\begin{aligned}
\infty&>\int_0^{t_0}\int_\rho^{t_0} \tau^{-\frac{2N-2}{N-2}} A\left(\tau^{-\frac{1}{N-2}}\right) d\tau d\rho\\[0.2cm]
&=\rho \int_\rho^{t_0} \tau^{-\frac{2N-2}{N-2}} A\left(\tau^{-\frac{1}{N-2}}\right) d\tau\Bigg|_{\rho=0}^{\rho=t_0}+\int_0^{t_0}\rho^{-\frac{N}{N-2}} A\left(\rho^{-\frac{1}{N-2}}\right) d\rho\\[0.2cm]
&=\int_0^{t_0}\rho^{-\frac{N}{N-2}} A\left(\rho^{-\frac{1}{N-2}}\right) d\rho\\[0.2cm]
&=(N-2)\int_{r_0}^\infty r A(r) dr,
\end{aligned}
$$
where, in the last equality, we use the change of variable $\rho^{-\frac{1}{N-2}}=r$. The above estimate raises a contradiction since $\ds \int_1^\infty tA(t) dt=\infty$. Thus, \eqref{ina1} has no positive solutions $w\in C^2(\R^N\setminus B_{r_0})$.  

(ii) Replacing in the following $r_0$ with $r_0+1$, we may assume $r_0>1$. With the change of variable $t=\log r$ and $W(t)=\overline w(r)$ we deduce from \eqref{ina2} that 
$$
-W''(t)\geq e^{2t}A(e^t)g(W(t)) \quad\mbox{ for all }t\geq t_0:=\log r_0.
$$
In particular $W$ is concave on $[t_0, \infty)$, so $W(t)\leq c(t+1)$ for all $t\geq t_0$, for some $c>0$. Since $g$ is nonincreasing we deduce 
$$
-W''(t)\geq e^{2t}A(e^t)g(c(t+1))\quad\mbox{ for all } t\geq t_0.
$$ 
Note that $\ds \liminf_{t\to \infty} e^{2t} A(e^t)g(c(t+1))>0$ implies the existence of $C>0$ and $\tau\geq t_0$ such that
$$
-W''(t)\geq C>0\quad\mbox{ for all } t\geq \tau.
$$ 
Integrating twice over $[\tau, t]$ we deduce
$$
-W(t)\geq \frac{C}{2}(t-\tau)^2-W'(\tau)(t-\tau)-W(\tau) \quad\mbox{ for all }t\geq \tau.
$$
This yields $\ds \lim_{t\to \infty}W(t)=-\infty$, which is a contradiction since $W$ is positive. 
\end{proof}

\begin{lemma}\label{com1}
Let $\Omega\subset \R^N$ be an open and bounded set and let $\Psi\in C(\Omega)$, $g\in C(0, \infty)$ be positive functions such that $g$ is nonincreasing.

Let also  $a,b\in C(\partial \Omega)$ be such that 
$$
a,b\geq 0, \quad (a(x), b(x))\neq (0,0)\quad\mbox{ for all }x\in \partial\Omega.
$$ 

Assume $w_1, w_2\in C^2(\Omega)\cap C^1(\overline \Omega)$ satisfy:
\begin{equation}\label{om1}
\Delta w_1+\Psi(x) g(w_1)\leq  \Delta w_2+\Psi(x) g(w_2)\quad\mbox{ in }\Omega,
\end{equation}
\begin{equation}\label{om2}
a(x) w_1+b(x)\frac{\partial w_1}{\partial \nu} \geq a(x) w_2+b(x)\frac{\partial w_2}{\partial \nu} \quad\mbox{ on }\partial\Omega,
\end{equation}
where $\nu$ denotes the outer unit normal vector at $\partial \Omega$. 
Then, $w_1\geq w_2$ in $\overline\Omega$.
\end{lemma}
\begin{proof}
Suppose the set $\Sigma:=\{x\in \overline \Omega: w_2(x)>w_1(x)\}$ is nonempty and let $w=w_2-w_1$. Let us note first that $w$ is subharmonic on $\Sigma$ since by \eqref{om1} we have
\begin{equation}\label{om3}
-\Delta w=\Delta w_1-\Delta w_2\leq \Psi(x)\Big( g(w_2)- g(w_1)\Big)<0 \quad\mbox{ in }\Sigma.
\end{equation}
Since $\Omega$ (and thus $\Sigma$) is bounded, there exists $x_0\in \overline \Sigma$ such that $w(x_0)=\max_{\overline\Omega}w>0$. 

\noindent{\bf Case 1.} $x_0\in \Omega$. Then, $x_0$ is an interior point of $\Omega$ and thus $-\Delta w(x_0)\geq 0$. This contradicts \eqref{om3}.

\noindent{\bf Case 2.} $x_0\in \partial \Omega$. Let us note that $w=w_2-w_1$ is subharmonic in $\Sigma$ which achieves its maximum at the boundary point $x_0\in \partial \Sigma\cap \partial \Omega$. Thus, by the Hopf boundary point lemma (see \cite[Section 6.4.2]{E2010}) we have 
$$
\frac{\partial w}{\partial \nu} (x_0)>0\Longrightarrow \frac{\partial w_2}{\partial \nu} (x_0)>\frac{\partial w_1}{\partial \nu} (x_0).
$$
Now, the above inequality together with $w_2(x_0)>w_1(x_0)$ contradict \eqref{om2} at $x=x_0$.
Hence, in both the above cases we raised a contradiction, which proves that $\Sigma$ is empty and thus $w_1\geq w_2$ in $\Omega$. 
\end{proof}
The following result is a counterpart of Lemma \ref{com1} for the case of exterior domains.

\begin{lemma}\label{com2}
Let $K\subset \R^N$ be a compact set and let $\Psi\in C(\R^N\setminus K)$, $g\in C(0, \infty)$ be positive functions.
Let $a,b\in C(\partial K)$ be such that 
$$
a,b\geq 0, \quad (a(x), b(x))\neq (0,0)\quad\mbox{ for all }x\in \partial K.
$$ 
Assume $w_1, w_2\in C^2(\R^N\setminus K )\cap C^1(\overline{\R^N\setminus K})$ satisfy:
\begin{equation}\label{om5}
\Delta w_1+\Psi(x) g(w_1)\leq \Delta w_2+\Psi(x) g(w_2)\quad\mbox{ in }\R^N\setminus K,
\end{equation}
\begin{equation}\label{om6}
a(x) w_1+b(x)\frac{\partial w_1}{\partial \nu} \geq a(x) w_2+b(x)\frac{\partial w_2}{\partial \nu} \quad\mbox{ on }\partial K,
\end{equation}
where $\nu$ denotes the outer unit normal vector at $\partial (\R^N\setminus K)$ and 
\begin{equation}\label{om7}
\lim_{|x|\to \infty} w_1(x)\geq \lim_{|x|\to \infty} w_2(x). 
\end{equation}
Then, $w_1\geq w_2$ in $\R^N\setminus K$.
\end{lemma}
\begin{proof} Let $\varepsilon\in (0,1)$. By \eqref{om7} we can find $R>0$ large enough such that $K\subset B_R$ and 
\begin{equation}\label{kkj}
w_1+\varepsilon\geq w_2 \quad\mbox{ in }\R^N\setminus B_R.
\end{equation} 
Let $\Omega=B_R\setminus K$. From \eqref{om5}, \eqref{om6} and \eqref{om7} we deduce that $\widetilde w_1=w_1+\varepsilon$ and $\widetilde w_2=w_2$ satisfy 
$$
\Delta \widetilde w_1+\Psi(x) g(\widetilde w_1)\leq  \Delta \widetilde w_2+\Psi(x) g(\widetilde w_2)\quad\mbox{ in }\Omega,
$$
and
$$
\widetilde a(x) \widetilde w_1(x)+\widetilde b(x)\frac{\partial \widetilde w_1}{\partial \nu} (x)\geq  \widetilde a(x) \widetilde w_2(x)+\widetilde b(x)\frac{\partial \widetilde w_2}{\partial \nu} (x)\quad\mbox{ on }\partial\Omega,
$$
where  $\widetilde a, \widetilde b\in C(\partial \Omega)$ are given by 
$$
\widetilde a(x)=
\begin{cases}
a(x)&\quad\mbox{ if }x\in \partial K\\[0.1cm]
1&\quad\mbox{ if }x\in \partial B_R
\end{cases}
\quad \mbox{ and }\quad 
\widetilde b(x)=
\begin{cases}
b(x)&\quad\mbox{ if }x\in \partial K\\[0.1cm]
0&\quad\mbox{ if }x\in \partial B_R
\end{cases}
.
$$
By Lemma \ref{com1} we deduce $\widetilde w_1\geq \widetilde w_2$ in $\Omega$ and thus, by the choice of $R>0$ in \eqref{kkj}, we have $w_1+\varepsilon\geq w_2$ in $\R^N\setminus K$. Passing to the limit with $\varepsilon\to 0$, we derive $w_1\geq w_2$ in $\R^N\setminus K$.
\end{proof}

\begin{proposition}\label{a2}

Let $K\subset \R^N$ $(N\geq 3)$  be a compact set which satisfies conditions $(K1)$ and $(K2)$.  Let also $\Psi\in C^{0,\gamma}_{loc}(\R^N\setminus \{0\})$, $\gamma\in (0,1)$, and $g\in C(0, \infty)$ be positive functions such that:
\begin{enumerate}
\item[{\rm (i)}] $g\in C^1(0,\infty)$ is nonincreasing; 
\item[{\rm (ii)}] $A(t):=\ds \sup_{|x|=t} \Psi(x)$ satisfies $\ds \int_1^\infty tA(t) dt<\infty$.
\end{enumerate}
Then, there exists a unique $w\in C^{2,\gamma}_{loc}(\overline{\R^N\setminus K})$  such that
\begin{equation}\label{ina3}
\begin{cases}
-\Delta w  = \Psi(x) g(w) & \quad\mbox{ in }\R^N\setminus K , \\[0.2cm]
\ds \;\;\; w>0 & \quad\mbox{ in }\overline{\R^N\setminus K} , \\[0.2cm]
\ds \;\;\; \frac{\partial w}{\partial \nu}  =0  & \quad\mbox{ on }\partial K , \\[0.2cm]
\;\;\; w(x)\to 0 &  \quad\mbox{ as }|x|\to \infty.
\end{cases}
\end{equation}

\end{proposition}

\begin{proof} Our approach relies on a version of the sub- and supersolution method as described in Theorem 2.1 in \cite[Section 3.2]{P1992}. We divide the proof into four steps as follows:
\begin{itemize} 
\item Step 1 discusses the existence of a supersolution. 
\item Step 2 and Step 3 detail the construction of a subsolution to \eqref{ina3}. This is a more delicate issue due to the fact that $g$ is nonincreasing, and thus may be singular around the origin. Precisely, in Step 2 we construct a subsolution which satisfies a homogeneous Dirichlet boundary condition on $\partial K$. Then, in Step 3 below we construct a proper subsolution to \eqref{ina3} which is a $C^2$ function up to (and including) the boundary of $K$. 
\item In Step 4 we deduce the existence of a $C^{2,\gamma}_{loc}(\overline{\R^N\setminus K})$ solution to \eqref{ina3}.  
\end{itemize}

\medskip

\noindent {\bf Step 1:} There exists $W\in C^2(\overline{\R^N\setminus K})$ a positive supersolution of  \eqref{ina3}, that is, $W$ satisfies
\begin{equation}\label{ina4}
\begin{cases}
-\Delta W  \geq  \Psi(x) g(W) & \quad\mbox{ in }\R^N\setminus K , \\[0.1cm]
\ds \;\;\; W>0 & \quad\mbox{ in }\overline{\R^N\setminus K} , \\[0.1cm]
\ds \;\;\; \frac{\partial W}{\partial \nu}  \geq 0  & \quad\mbox{ on }\partial K , \\[0.2cm]
\;\;\; W(x)\to 0 &  \quad\mbox{ as }|x|\to \infty.
\end{cases}
\end{equation}

By $(K1)$ there exists $r_0>0$ such that $B_{r_0}\subset K$. Using the integral condition on $A$ and a direct integration by parts, for $R>r_0$ we have
$$
\begin{aligned}
\int_{r_0}^R  t^{1-N}\int_{r_0}^t \tau^{N-1} A(\tau) d\tau&=-\frac{t^{2-N}}{N-2}\int_{r_0}^t \tau^{N-1} A(\tau) d\tau \Bigg|_{t=r_0}^{t=R}+\frac{1}{N-2}\int_{r_0}^R tA(t) dt\\[0.2cm]
&\leq \frac{1}{N-2}\int_{r_0}^R tA(t) dt\\[0.2cm]
&\leq \frac{1}{N-2}\int_{r_0}^\infty tA(t) dt<\infty.
\end{aligned}
$$
Thus,  we may define
\begin{equation}\label{fii0}
Z(x)=\int_{|x|}^\infty t^{1-N}\int_{r_0}^t \tau^{N-1} A(\tau) d\tau\quad\mbox{ for all }x\in \R^N\setminus B_{r_0}.
\end{equation}
Also, $Z$ satisfies
\begin{equation}\label{ffi1}
\begin{cases}
-\Delta Z(x) =A(|x|) &\quad\mbox{ in } \R^N\setminus B_{r_0},\\[0.1cm]
 \;\; \; \;\; Z(x)\to 0 &\quad\mbox{ as } |x|\to \infty.
\end{cases}
\end{equation} 
We further define
$W:\R^N\setminus K\to \R$ by 
\begin{equation}\label{ffi01}
\int_0^{W(x)} \frac{1}{g(t)}dt=Z(x)\quad\mbox{ for all }x\in \R^N\setminus K.
\end{equation}
Then $W\in C^2(\overline{\R^N\setminus K})$, $W(x)\to 0$ as $|x|\to\infty$,  and 
\begin{equation}\label{ffi2}
\nabla W(x)=g(W(x))\nabla Z(x) \quad\mbox{ for all } x\in \overline{\R^N\setminus K}.
\end{equation} 
Using \eqref{ffi1}, \eqref{ffi2} and $g'\leq 0$, we have
$$
\begin{aligned}
-\Delta W(x)&=-{\rm div}(\nabla W(x))\\[0.1cm]
&=-g(W(x))\Delta Z(x)-g'(W(x))\nabla W(x)\cdot \nabla Z(x)\\[0.1cm]
&=A(|x|)g(W(x))-g'(W(x)) g(W(x))|\nabla Z(x)|^2\\[0.1cm]
&\geq A(|x|)g(W(x)) \quad \mbox{(since $g'\leq 0$)} \\[0.1cm]
&\geq \Psi(x) g(W(x)) \quad \mbox{ for all }x\in \R^N\setminus K.
\end{aligned}
$$
Finally, using \eqref{eqbd}, \eqref{fii0} and \eqref{ffi2}, for all $x\in \partial K$ we find
$$
\begin{aligned}
\frac{\partial W}{\partial \nu}(x)&=\nabla W(x)\cdot \nu(x)=g(W(x))\nabla Z(x) \cdot \nu(x)\\[0.1cm]
&=-|x|^{-N}\Big(\int^{|x|}_{r_0} t^{N-1}A(t) dt\Big) \nu(x)\cdot x\geq 0.
\end{aligned}
$$
Thus, $W$ defined by \eqref{ffi01} is a supersolution of \eqref{ina3}.

\noindent {\bf Step 2:} Fix $\delta>0$. There exists $U\in C^2(\overline{\R^N\setminus K})$ such that
\begin{equation}\label{ina5}
\begin{cases}
-\Delta U  = \Psi(x)  g(U+\delta) \; , \; U>0& \quad\mbox{ in }\R^N\setminus K , \\[0.1cm]
\ds \;\;\; U=0 & \quad\mbox{ on }\partial K , \\[0.1cm]
\ds \;\;\; \frac{\partial U}{\partial \nu}  < 0  & \quad\mbox{ on }\partial K , \\[0.2cm]
\;\;\; U(x)\to 0 &  \quad\mbox{ as }|x|\to \infty.
\end{cases}
\end{equation}

Let  $n_0\geq 1$ be a positive integer such that $K\subset B_{n_0}$. For all $n\geq n_0$ consider the problem
 \begin{equation}\label{ina6}
\begin{cases}
-\Delta u_n  = \Psi(x) g(u_n+\delta) \; , \; u_n>0& \quad\mbox{ in }B_n\setminus K , \\[0.1cm]
\ds \;\;\; u_n=0 & \quad\mbox{ on }\partial K , \\[0.1cm]
\ds \;\;\; u_n=W  & \quad\mbox{ on }\partial B_n.  
\end{cases}
\end{equation}
The existence of $u_n\in C^2(\overline{B_n\setminus K})$ follows from the sub and supersolution method once we observe that $\underline u_n\equiv 0$ and $\overline u_n=W$ are respectively sub and supersolutions to \eqref{ina6}. We should point out that $g$ may be singular around zero but the function $[0,\infty)\ni t\mapsto g(t+\delta)$ is of class $C^1$ and this yields the existence of a $C^2$-solution in the sub and supersolution process. Next, we apply Lemma \ref{com1} for 
$$
\Omega=B_n\setminus K, \quad a\equiv 1, \quad b\equiv 0, \quad w_1=u_{n+1},  \quad w_2=u_n,
$$
to deduce $u_{n+1}\leq u_n$ in $B_n\setminus K$. 
We extend $u_n=W$ in $\R^N\setminus B_n$ and thus we have
\begin{equation}\label{ina7}
0\leq u_{n+1}\leq u_n\leq W\quad\mbox{ in }\R^N\setminus K \, ,\mbox{ for all }n\geq n_0.
\end{equation}
By the standard elliptic arguments (see \cite{ADN1, ADN2}) we deduce that $U(x)=\lim_{n\to \infty}u_n(x)$ satisfies $U\in C^2(\overline{\R^N\setminus K})$ and $-\Delta U=\Psi(x) g(U+\delta)$ in $\R^N\setminus K$. Clearly $U$ is not constant and thus the maximum principle yields $U>0$ in $\R^N\setminus K$. Also, for all $x\in \partial K$ one has $U(x)=\lim_{n\to \infty}u_n(x)=0$. Thus, any $x\in \partial K$ is a minimum point of $U$ and thus, by the Hopf boundary point lemma (see \cite[Section 6.4.2]{E2010}) one has $\frac{\partial U}{\partial \nu}<0$ on $K$. Finally, from \eqref{ina7} we deduce $U\leq W$ in $\R^N\setminus K$, so $U(x)\to 0$ as $|x|\to \infty$. Hence, $U$ satisfies \eqref{ina5}.

\medskip

\noindent {\bf Step 3:} There exists $V\in C^2(\overline{\R^N\setminus K})$ such that
\begin{equation}\label{ina8}
\begin{cases}
-\Delta V  = \Psi(x) g(V+\delta) & \quad\mbox{ in }\R^N\setminus K , \\[0.1cm]
\ds \;\;\; V>0 & \quad\mbox{ in } \overline{\R^N\setminus K}, \\[0.1cm]
\ds \;\;\; \frac{\partial V}{\partial \nu}  = 0  & \quad\mbox{ on }\partial K , \\[0.2cm]
\;\;\; V(x)\to 0 &  \quad\mbox{ as }|x|\to \infty.
\end{cases}
\end{equation}

As before, let  $n_0\geq 1$ be a positive integer such that $K\subset B_{n_0}$. For all $n\geq n_0$ consider the problem
 \begin{equation}\label{ina9}
\begin{cases}
-\Delta v_n  = \Psi(x) g(v_n+\delta) \; , \; v_n>0& \quad\mbox{ in }B_n\setminus K , \\[0.1cm]
\ds \;\;\; \frac{\partial v_n}{\partial \nu}=0 & \quad\mbox{ on }\partial K , \\[0.1cm]
\ds \;\;\; v_n=W  & \quad\mbox{ on }\partial B_n.  
\end{cases}
\end{equation}
Notice that \eqref{ina9} is the counterpart of \eqref{ina6} in which the homogeneous Dirichlet boundary condition on $\partial K$ is replaced by a Neumann one. 
To check the existence of a solution to \eqref{ina9} we first observe that for all $n\geq n_0$, $\underline v_n=U$ (where $U$ satisfies \eqref{ina5}) is a subsolution to \eqref{ina9}  while $\overline v_n=W$ is a supersolution, where $W$ satisfies \eqref{ina4}. As in Step 2 above, we now derive the existence of a solution $v_n\in C^2(\overline{B_n\setminus K})$ to \eqref{ina9} such that $U\leq v_n\leq W$ in $B_n\setminus K$. By the maximum principle we have $v_n>0$ in $B_n\setminus K$. Apply now Lemma \ref{com1} with $\Omega=B_n\setminus K$ and $a(x)$, $b(x)$ given by
\begin{equation}\label{oo}
a(x)=
\begin{cases}
0&\quad\mbox{ if }x\in \partial K\\[0.1cm]
1&\quad\mbox{ if }x\in \partial B_n
\end{cases}
\quad \mbox{ and }\quad 
b(x)=
\begin{cases}
1&\quad\mbox{ if }x\in \partial K\\[0.1cm]
0&\quad\mbox{ if }x\in \partial B_n
\end{cases}
,
\end{equation} 
we find $v_{n+1}\leq v_n$ in $B_n\setminus K$. Extending $v_n=W$ on $\R^N\setminus B_n$, we have that 
\begin{equation}\label{ina10}
U \leq v_{n+1}\leq u_n\leq W\quad\mbox{ in }\R^N\setminus K\, ,\mbox{ for all }n\geq n_0.
\end{equation}
Set $V(x)=\lim_{n\to \infty}v_n(x)$. 
By standard elliptic arguments (see \cite{ADN1,ADN2}) we deduce that $\{v_n\}_{n\geq n_0}$ is bounded in $C^{2,\gamma}_{loc}(\overline{\R^N\setminus K})$ and thus  $V$ satisfies 
\begin{equation}\label{ina11}
-\Delta V  = \Psi(x) g(V+\delta)  \quad\mbox{ in }\R^N\setminus K\quad\mbox{ and }\quad 
\frac{\partial V}{\partial \nu}  = 0   \quad\mbox{ on }\partial K.
\end{equation}
Passing to the limit in \eqref{ina10} we find $U\leq V\leq W$, so $V>0$ in $\R^N\setminus K$ and $V(x)\to 0$ as $|x|\to \infty$. It remains to check that $V>0$ on $\partial K$ which can be done as in Step 2 above: if $x_0\in \partial K$ satisfies $V(x_0)=0$, then $x_0$ is a minimum point of $V$ at which, by Hopf boundary point lemma (see \cite[Section 6.4.2]{E2010}), one has $\frac{\partial V}{\partial \nu} (x_0) < 0$. This clearly contradicts the boundary condition in \eqref{ina11} and proves $V>0$ on $\partial K$. Thus, $V$ satisfies \eqref{ina8}.

\medskip

\noindent {\bf Step 4:} Construction of a solution to \eqref{ina3}.

For any $n\geq n_0$ consider the problem 
 \begin{equation}\label{ina12}
\begin{cases}
-\Delta w_n  =\Psi(x) g(w_n) \; , \; w_n>0& \quad\mbox{ in }B_n\setminus K , \\[0.1cm]
\ds \;\;\; \frac{\partial w_n}{\partial \nu}=0 & \quad\mbox{ on }\partial K , \\[0.1cm]
\ds \;\;\; w_n=W  & \quad\mbox{ on }\partial B_n.  
\end{cases}
\end{equation}
Then,  $\underline w_n=V$ (where $V$ satisfies \eqref{ina8}) is a subsolution to \eqref{ina12}  while $\overline w_n=W$ is a supersolution, where $W$ satisfies \eqref{ina4}. Thus, the sub and supersolution method yields the existence of a solution $w_n\in C^{2, \gamma}(\overline{B_n\setminus K})$ to \eqref{ina12} such that $V\leq w_n\leq W$ in $B_n\setminus K$. We should note at this stage that $V>0$ in $\overline{B_n\setminus K}$ yields $w_n\in C^{2, \gamma}(\overline{B_n\setminus K})$. By Lemma \ref{com1} for $\Omega=B_n\setminus K$ and $a(x)$, $b(x)$ given by \eqref{oo}, we deduce $w_n\leq w_{n+1}$ in $B_n\setminus K$. 
Hence, extending $w_n=W$ on $\R^N\setminus B_n$ we have
\begin{equation}\label{ina13}
V \leq w_{n+1}\leq w_n\leq W\quad\mbox{ in }\R^N\setminus K\, ,\mbox{ for all }n\geq n_0.
\end{equation}
Set $w(x)=\lim_{n\to \infty}w_n(x)$. From \eqref{ina13} one has $V\leq w\leq W$ in $\R^N\setminus K$, so $w>0$ in $\overline{\R^N\setminus K}$ and $w(x)\to 0$ as $|x|\to \infty$. Now, standard elliptic arguments allow us to conclude that $w\in C^{2,\gamma}_{loc}(\overline{\R^N\setminus K})$ is a solution to \eqref{ina3}. The uniqueness of the solution to \eqref{ina3} follows from Lemma \ref{com2}.
\end{proof}

\begin{lemma}\label{a3}

Let $s>0$ and $K\subset \R^N$ $(N\geq 3)$ be a compact set that satisfies $(K1)$ and $(K2)$. Assume $\phi\in C^{0,\gamma}_{loc}(\overline{\R^N\setminus K})$, $\gamma\in (0,1)$, is positive and $\phi(x)\simeq |x|^{-\alpha}$ for some $\alpha>0$. 
Then, the problem
\begin{equation}\label{inaa1}
\begin{cases}
-\Delta w  = \phi(x)   w^{-s} & \quad\mbox{ in }\R^N\setminus K , \\[0.2cm]
\ds \;\;\; w>0 & \quad\mbox{ in }\overline{\R^N\setminus K} , \\[0.2cm]
\ds \;\;\; \frac{\partial w}{\partial \nu}  =0  & \quad\mbox{ on }\partial K , \\[0.2cm]
\;\;\; w(x)\to 0 &  \quad\mbox{ as }|x|\to \infty,
\end{cases}
\end{equation}
has a solution $w\in C^{2}(\overline{\R^N\setminus K})$ if and only if $\alpha>2$. Moreover, we have:

\begin{enumerate}
\item[{\rm (i)}] If $\alpha>2$ then \eqref{inaa1} has a unique solution $w\in C^{2,\gamma}_{loc}(\overline{\R^N\setminus K})$.
\item[{\rm (ii)}] $w$ satisfies 
\begin{equation}\label{lueq1}
w(x)\simeq 
\begin{cases}
|x|^{-\frac{\alpha-2}{1+s}} &\quad\mbox{ if } 2<\alpha<N+s(N-2),\\[0.1in]
|x|^{2-N}\log^{\frac{1}{1+s}} &\quad\mbox{ if } \alpha=N+s(N-2),\\[0.1in]
|x|^{2-N} &\quad\mbox{ if } \alpha>N+s(N-2).
\end{cases}
\end{equation}
\end{enumerate}
\end{lemma}
If $2<\alpha<N+s(N-2)$, Lemma \ref{a3} shows that the asymptotic behavior of the unique solution $w$ to \eqref{inaa1} depends on $\alpha$. Further, if $\alpha\geq N+s(N-2)$, then, the asymptotic behavior of the unique solution $w$ is independent on $\alpha$.

\begin{proof} (i) Let $A(t)=t^{-\alpha}$ and $g(t)=t^{-s}$. By Lemma \ref{a1} we deduce that \eqref{inaa1} has no positive solutions if $\alpha\le 2$. On the other hand, if $\alpha>2$, Proposition \ref{a2} yields  the existence of a solution $w\in C^{2,\gamma}_{loc}(\overline{\R^N\setminus K})$ to \eqref{inaa1}. This completes the proof of part (i).

(ii1) Assume $2<\alpha<N+s(N-2)$. The function $z(x)=|x|^{-\frac{\alpha-2}{1+s}}$ satisfies
$$
-\Delta z(x)=m |x|^{-\alpha} z^{-s} \quad\mbox{ in }\R^N\setminus K,
$$
where
$$
m=\frac{\alpha-2}{1+s}\Big(N-2-\frac{\alpha-2}{1+s}  \Big)>0.
$$
We now choose a small constant $0<c<m^{-\frac{1}{1+s}}$ so that $cz\leq w$ on $\partial K$. Thus, $w$ and $cz$ satisfy:
\begin{equation}\label{eqwq1}
\begin{aligned}
\Delta w +|x|^{-\alpha} w^{-s} &=0\leq \Delta (cz)+|x|^{-\alpha} (cz)^{-s}  \quad\mbox{ in }\R^N\setminus K,\\[0.1cm]
w&\geq cz  \quad\mbox{ on }\partial K,\\[0.1cm]
w(x), \, cz(x)&\to 0 \quad\mbox{ as } |x|\to \infty.
\end{aligned}
\end{equation}
We now apply Lemma \ref{com2} with $a\equiv 1$, $b\equiv 0$ on $\partial K$ to deduce $w\geq cz$ in $\R^N\setminus K$. This yields the lower bound in \eqref{lueq1}. For the upper bound in \eqref{lueq1}, let $C>m^{-\frac{1}{1+s}}$ be large enough such that $Cz\geq w$ on $\partial K$. Then:
\begin{equation}\label{eqwq2}
\begin{aligned}
\Delta w +|x|^{-\alpha} w^{-s} &=0\geq \Delta (Cz)+|x|^{-\alpha} (Cz)^{-s}  \quad\mbox{ in }\R^N\setminus K,\\[0.1cm]
w&\leq Cz  \quad\mbox{ on }\partial K,\\[0.1cm]
w(x), \, Cz(x)&\to 0 \quad \mbox{ as } |x|\to \infty.
\end{aligned}
\end{equation}
As above and with the help of Lemma \ref{com2} we find $Cz\geq w$ in $\R^N\setminus K$ and this concludes the proof of \eqref{lueq1} in the case $2<\alpha<N+s(N-2)$.

(ii2)  Assume $\alpha=N+s(N-2)$. The function $z(x)=|x|^{2-N}\log^{\frac{1}{1+s}}\Big(\frac{|x|}{r_0}\Big)$ satisfies
$$
-\Delta z(x)=\frac{N-2}{1+s} |x|^{-N}\log^{-\frac{s}{1+s}} \Big(\frac{|x|}{r_0}\Big)+\frac{s}{(1+s)^2} |x|^{-N}\log^{-\frac{1+2s}{1+s}} \Big(\frac{|x|}{r_0}\Big) \quad\mbox{ in }\R^N\setminus K.
$$
Hence,
$$
-\Delta z(x)\simeq  |x|^{-N}\log^{-\frac{s}{1+s}} \Big(\frac{|x|}{r_0}\Big)\simeq |x|^{-\alpha} z^{-s} \quad\mbox{ in }\R^N\setminus K.
$$
The above estimates show that $z$ satisfies \eqref{eqwq1} for some constant $c>0$ and then, by Lemma \ref{com2} with $a\equiv 1$, $b\equiv 0$ on $\partial K$, we find $w\geq cz$ in $\R^N\setminus K$. In a similar way, for some large constant $C>0$ we have $Cz\geq w$ in $\R^N\setminus K$ and this yields \eqref{lueq1} in the case $\alpha=N+s(N-2)$.

(ii3) Assume $\alpha>N+s(N-2)$ and let $C_0=r_0^{\alpha-N-s(N-2)}>0$. Define  
$$
z(x)=|x|^{2-N}-C_0|x|^{2+s(N-2)-\alpha}.
$$ 
Then $z>0$ in $\R^N\setminus K$ and
$$
\begin{aligned}
-\Delta z(x)&=C_0\big(\alpha-2-s(N-2)\big)\big(\alpha-N-s(N-2)\big)|x|^{s(N-2)-\alpha}\\
&\simeq |x|^{s(N-2)-\alpha}\\
&\simeq |x|^{-\alpha}z^{-s} \quad\mbox{ in }\R^N\setminus K.
\end{aligned}
$$
In particular, this means that for some $C>c>0$ we have
$$
\Delta (cz)+|x|^{-\alpha} (cz)^{-s} \geq 0\geq \Delta (Cz)+|x|^{-\alpha} (Cz)^{-s} \quad\mbox{ in }\R^N\setminus K.
$$
By taking $C>0$ sufficiently large, and $c>0$ sufficiently small, we have $Cz\geq w\geq cz$ on $\partial K$. Thus, $cz$ and $Cz$ satisfy \eqref{eqwq1} and \eqref{eqwq2} respectively. We may now invoke Lemma \ref{com2} to deduce $cz\leq w\leq Cz$ in $\R^N\setminus K$, which is equivalent to \eqref{lueq1}.
\end{proof}

\begin{lemma}\label{a4}
Let $K\subset \R^N$ $(N\geq 3)$  be a compact set that satisfies the conditions $(K1)$ and $(K2)$. Assume $h\in C^{0,\gamma}_{loc}(\overline{\R^N\setminus K})$, $\gamma\in (0,1)$, is positive and satisfies 
$$
h(x)\simeq |x|^{-\beta} \quad\mbox{ in }\R^N\setminus K\; \mbox{ for some }\beta>2.
$$
Then, there exists a unique function  $w\in C^{2,\gamma}_{loc}(\overline{\R^N\setminus K})$  such that
\begin{equation}
\label{corw1}
\begin{cases}
-\Delta w  = h(x)   & \quad\mbox{ in }\R^N\setminus K , \\[0.2cm]
\ds \;\;\; w>0 & \quad\mbox{ in }\overline{\R^N\setminus K} , \\[0.2cm]
\ds \;\;\; \frac{\partial w}{\partial \nu}  =0  & \quad\mbox{ on }\partial K , \\[0.2cm]
\;\;\; w(x)\to 0 &  \quad\mbox{ as }|x|\to \infty.
\end{cases}
\end{equation}
Furthermore, $w$ satisfies 
\begin{equation}
\label{corw2}
w(x)\simeq \begin{cases}
|x|^{2-\beta} &\quad \mbox{ if }2<\beta<N\\
|x|^{2-N}\log \Big(\frac{|x|}{r_0}\Big)  &\quad \mbox{ if } \beta=N\\
|x|^{2-N} &\quad \mbox{ if } \beta>N
\end{cases}
\quad\mbox{ in }\R^N\setminus K.
\end{equation}
\end{lemma}
\begin{proof} The existence and uniqueness of a solution $w\in C^{2,\gamma}_{loc}(\overline{\R^N\setminus K})$ to \eqref{corw1} follows from Proposition \ref{a2} in which we take $g\equiv 1$ and $\Psi\equiv h$. We next focus on the behaviour of $w$ given in \eqref{corw2}.
Let 
$$
z(x)=
\begin{cases}
|x|^{2-\beta} &\quad \mbox{ if }2<\beta<N,\\[0.2cm]
|x|^{2-N}\log \Big(\frac{|x|}{r_0}\Big)  &\quad \mbox{ if } \beta=N,\\[0.3cm]
|x|^{2-N}-r_0^{\beta-N} |x|^{2-\beta} &\quad \mbox{ if } \beta>N.
\end{cases}
$$
Then $z>0$ in $\overline{\R^N\setminus K}$ and
$$
-\Delta z(x)=
\begin{cases}
(\beta-2)(N-\beta) |x|^{-\beta} &\quad \mbox{ if }2<\beta<N,\\[0.2cm]
(N-2) |x|^{-N}  &\quad \mbox{ if } \beta=N,\\[0.3cm]
r_0^{N-\beta} (\beta-2)(\beta-N) |x|^{-\beta} &\quad \mbox{ if } \beta>N.
\end{cases}
$$
We fix two constants $C>c>0$ such that $\underline w=c z$ and $\overline w=Cz$ satisfy 
\begin{equation}
\label{corw3}
\underline w\leq w\leq \overline w\quad\mbox{ on }\partial K
\end{equation}
and
\begin{equation}
\label{corw4}
\Delta \underline w+h(x)\leq 0=\Delta w +h(x)\leq \Delta \overline w +h(x) \quad\mbox{ in }\R^N\setminus K.
\end{equation}
By Lemma \ref{com2} in which we take $\Psi\equiv h$, $g\equiv 1$, $a\equiv 1$ and $b\equiv 0$ we deduce  
$\underline w(x)\leq w\leq \overline w$ in $\R^N\setminus K$. Hence, $w\simeq z$ in $\R^N\setminus K$, which proves \eqref{corw2}.
\end{proof}

A similar result to Lemma \ref{a4} is stated below.
\begin{lemma}\label{a5}
Let $K\subset \R^N$ $(N\geq 3)$  be a compact set that satisfies conditions $(K1)$ and $(K2)$. Assume $h\in C^{0,\gamma}_{loc}(\overline{\R^N\setminus K})$, $\gamma\in (0,1)$, is positive and  satisfies 
$$
h(x)\simeq |x|^{-N}\log^{-\theta}\Big(\frac{|x|}{r_0}\Big)  \quad\mbox{ in }\R^N\setminus K,\; \mbox{ for some }\theta>0.
$$
Then, there exists a unique function  $w\in C^{2,\gamma}_{loc}(\overline{\R^N\setminus K})$  that satisfies \eqref{corw1}. 
Furthermore, we have
\begin{equation}
\label{corw5}
w(x)\simeq \begin{cases}
|x|^{2-N} \log^{1-\theta}\Big(\frac{|x|}{r_0}\Big) &\quad \mbox{ if }0<\theta<1\\
|x|^{2-N}\log\bigg( C_0 \log\Big(\frac{|x|}{r_0}\Big)\bigg)  &\quad \mbox{ if } \theta=1\\
|x|^{2-N} &\quad \mbox{ if } \theta>1
\end{cases}
\quad\mbox{ in }\R^N\setminus K,
\end{equation}
where $C_0>1$ is a large constant such that $|x|>r_0e^{1/C_0}$ for all $x\in \R^N\setminus K$.
\end{lemma}
\begin{proof} As above, the existence and uniqueness of a solution $w\in C^{2,\gamma}_{loc}(\overline{\R^N\setminus K})$ to \eqref{corw1} follows from Proposition \ref{a2}. We next derive the asymptotic behavior in \eqref{corw5}.
\medskip

\noindent{\bf Case 1:} $0<\theta <1$. Let $z(x)=|x|^{2-N} \log^{1-\theta}\Big(\frac{|x|}{r_0}\Big)$. Then, by a direct calculation we find
$$
-\Delta z(x)=(N-2)(1-\theta)|x|^{-N} \log^{-\theta}\Big(\frac{|x|}{r_0}\Big)+\theta(1-\theta)|x|^{-N} \log^{-1-\theta}\Big(\frac{|x|}{r_0}\Big)
$$
which shows that 
$$
-\Delta z(x)\simeq h(x)\quad \mbox{ in }\R^N\setminus K.
$$
Thus, for appropriate constants $C>c>0$ we have that $\underline w=cz$ and $\overline w=Cz$ satisfy \eqref{corw4}. A similar approach to that  in the proof of Lemma \ref{a4} now yields 
$$
w(x)\simeq z(x)\simeq |x|^{2-N} \log^{1-\theta}\Big(\frac{|x|}{r_0}\Big)\quad\mbox{ in } \R^N\setminus K.
$$

\noindent{\bf Case 2:} $\theta =1$. The function $z(x)=|x|^{2-N}\log\bigg( C_0 \log\Big(\frac{|x|}{r_0}\Big)\bigg)  $ satisfies
$$
-\Delta z(x)=(N-2) |x|^{-N}\log^{-1}\Big(\frac{|x|}{r_0}\Big)+|x|^{-N} \log^{-2}\Big(\frac{|x|}{r_0}\Big)\simeq h(x) 
\quad\mbox{ in } \R^N\setminus K.
$$
We proceed as above in order to deduce 
$$
w(x)\simeq z(x)\simeq |x|^{2-N}\log\bigg(C_0 \log\Big(\frac{|x|}{r_0}\Big)\bigg)  \quad\mbox{ in } \R^N\setminus K.
$$ 
\noindent{\bf Case 3:} $\theta >1$. Let 
$$
z(x)=|x|^{2-N}\bigg(1-C\log^{1-\theta}\Big(\frac{|x|}{r_0}\Big)\bigg),
$$ 
where $C>0$ is large enough so that $z>0$ in $\overline{\R^N\setminus K}$. Since $\theta>1$, we have 
$$
\begin{aligned}
-\Delta z(x)&=C(N-2)(\theta-1)|x|^{-N} \log^{-\theta}\Big(\frac{|x|}{r_0}\Big)+C\theta(\theta-1)|x|^{-N} \log^{-1-\theta}\Big(\frac{|x|}{r_0}\Big)\\[0.1cm]
&\simeq h(x) 
\quad\mbox{ in } \R^N\setminus K.
\end{aligned}
$$
With a similar argument to that in the proof of Lemma \ref{a4} we now deduce $
w(x)\simeq z(x)\simeq |x|^{2-N}$ in $\R^N\setminus K$. This concludes the proof of our Lemma.
\end{proof}

\section{Proof of Theorem \ref{th1}}
Assume $(u,v)$ is a positive solution of \eqref{GM0}.

(i) If $N=2$, by Lemma \ref{lest}(i) we have $u(x)\geq c$ in $\R^2\setminus B_R$, for some large $R>0$ so that $K\subset B_R$. Then, $v$ satisfies $-\Delta v\geq c^m v^{-s}$ in $\R^N\setminus B_R$. Now, by Lemma \ref{a1}(ii), this last inequality has no positive solutions.

(ii) If $N\geq 3$ and $m\leq \frac{2}{N-2}$, we apply Lemma \ref{lest}(ii) to obtain $u\geq C|x|^{2-N}$ in $\R^N\setminus B_R$, for some $R>0$. We use this estimate in the second equation of \eqref{GM0} to deduce that $v$ satisfies 
$$
-\Delta v\geq C^{m}|x|^{-m(N-2)}\quad\mbox{ in } \R^N\setminus B_R.
$$ 
We now apply Lemma \ref{a1}(i) with $A(x)=C^{m}|x|^{-m(N-2)}$ to raise a contradiction since $m(N-2)\leq 2$. 

(iii) Assume $N\geq 3$ and $0<p\leq \frac{N}{N-2}$.  Since $v(x)\to 0$ as $|x|\to \infty$, there exists $M>0$ such that $v(x)\leq M$ in $\R^N\setminus K$. Then, $u$ satisfies $-\Delta u\geq M^{-q} u^p$ in $\R^N\setminus K$. We now deduce from \cite[Theorem 2.1]{AS11} that this last inequality has no positive solutions if $0<p\leq \frac{N}{N-2}$.
\qed

\section{Proof of Theorem \ref{th2}}

Assume first that \eqref{GM0} has a positive solution with $u(x)\simeq |x|^{2-N}$ in $\R^N\setminus K$. By Theorem \ref{th1}(ii) we know that $m>\frac{2}{N-2}$.  Then, $v$ satisfies \eqref{inaa1} in Lemma \ref{a3}, where $\phi(x)=u(x)^m\simeq |x|^{-\alpha}$ and $\alpha=m(N-2)>2$. By this result, we have 
\begin{equation}\label{ps0}
v(x)\simeq 
\begin{cases}
\ds |x|^{2-N} & \quad \mbox{ if } m> s+\frac{N}{N-2} ,\\[0.2cm]
\ds |x|^{2-N}\log^{\frac{1}{1+s}} \Big(\frac{|x|}{r_0}\Big)  &\quad \mbox{ if } m= s+\frac{N}{N-2},\\[0.2cm]
\ds |x|^{-\frac{m(N-2)-2}{1+s}} &\quad \mbox{ if }\frac{2}{N-2}<m< s+\frac{N}{N-2}.
\end{cases}
\end{equation}
From the first equation of \eqref{GM0} we have
\begin{equation}\label{eqh1}
-\Delta u(x)\simeq u^pv^{-q}+|x|^{-k}\simeq |x|^{-p(N-2)}v(x)^{-q}+|x|^{-k} \quad \mbox{ in }\R^N\setminus K.
\end{equation}
\noindent{\bf Case 1:} $m>s+\frac{N}{N-2}$. From \eqref{ps0} and \eqref{eqh1} we have
\begin{equation}\label{eqh2}
-\Delta u(x)\simeq |x|^{-\beta} \quad \mbox{ in }\R^N\setminus K,
\end{equation}
where $\beta=\min\big\{ k, (p-q)(N-2)\big\}$. Using Lemma \ref{a4} we see now that $u(x)\simeq |x|^{2-N}$ in $\R^N\setminus K$  if and only if $\beta>N$, that is $k>N$ and $p>q+\frac{N}{N-2}$. We thus retrieve condition (i) in the statement of Theorem \ref{th1}.

\medskip

\noindent{\bf Case 2:} $\frac{2}{N-2}<m<s+\frac{N}{N-2}$. Then, from \eqref{ps0} and \eqref{eqh1} we have that $u$ satisfies \eqref{eqh2} where 
\begin{equation}\label{be}
\beta=\min\bigg\{ k, p(N-2)-q\frac{m(N-2)-2}{1+s}\bigg\}.
\end{equation}
Now, $u(x)\simeq |x|^{2-N}$ in $\R^N\setminus K$ and Lemma \ref{a4} yield $\beta>N$ and we retrieve condition (iii) in the statement of Theorem \ref{th1}.
\medskip

\noindent{\bf Case 3:} $m=s+\frac{N}{N-2}$. This case is more delicate as one has to handle the logarithmic terms in the asymptotic of $v(x)$ in \eqref{ps0}. Let $\varepsilon>0$. From \eqref{ps0} and \eqref{eqh1} we have 
\begin{equation}\label{eqh3}
-\Delta u\simeq |x|^{-(p-q)(N-2)} \log^{-\frac{q}{1+s}} \Big(\frac{|x|}{r_0}\Big)+ |x|^{-k}\quad\mbox{ in }\R^N\setminus K.
\end{equation}
In particular, since $\log^{-\frac{q}{1+s}}t\geq c_\varepsilon t^{-\varepsilon(N-2)}$ as $t\geq 1$, one has 
$$
-\Delta u(x)\geq C_\varepsilon  |x|^{-\beta} \quad \mbox{ in }\R^N\setminus K,
$$
where 
$$
\beta=\min\big\{ k, (p-q+\varepsilon)(N-2)\big\}.
$$
With the same method as in the proof of Lemma \ref{a4} we have 
$$
u(x)\geq C_\varepsilon
\begin{cases}
|x|^{2-\beta} &\quad \mbox{ if }2<\beta<N\\
|x|^{2-N}\log \Big(\frac{|x|}{r_0}\Big)  &\quad \mbox{ if } \beta=N\\
|x|^{2-N} &\quad \mbox{ if } \beta>N
\end{cases}
\quad\mbox{ in }\R^N\setminus K.
$$
Thus, $u(x)\simeq |x|^{2-N}$ in $\R^N\setminus K$ yields $\beta>N$ so that $k>N$ and $p-q+\varepsilon>\frac{N}{N-2}$. Since $\varepsilon>0$ was arbitrary, one must have $p\geq q+\frac{N}{N-2}$. 

If $p>q+\frac{N}{N-2}$, we retrieve condition (i) in the statement of Theorem \ref{th2}.
\smallskip

If $p=q+\frac{N}{N-2}$ we reassess the above estimates and from $k>N$ and \eqref{eqh3} we deduce
$$
-\Delta u\simeq |x|^{-N} \log^{-\frac{q}{1+s}} \Big(\frac{|x|}{r_0}\Big)\quad\mbox{ in }\R^N\setminus K.
$$
We combine the above asymptotic with the conclusion of Lemma \ref{a5} and $u(x)\simeq |x|^{2-N}$ in $\R^N\setminus K$ to derive $q>1+s$. This yields condition (ii) in the statement of Theorem \ref{th2}.

Conversely, assume now that one of conditions (i)-(iii) in the statement of Theorem \ref{th2} hold and let us prove the existence of a positive solution $(u,v)$ of \eqref{GM0} with $u(x)\simeq |x|^{2-N}$ in $\R^N\setminus K$. Define
\begin{equation}\label{aset}
\mathcal{A}=\Bigg\{(u, v)\in C^{0,\gamma}_{loc}(\overline{\R^N\setminus K})\times C^{0,\gamma}_{loc}(\overline{\R^N\setminus K}): \;
\begin{aligned}
D|x|^{2-N}&\leq u(x)\leq E |x|^{2-N}\\
F\psi(x)&\leq v(x)\leq G \psi(x)
\end{aligned} \Bigg\},
\end{equation}
where $\psi(x) $ is given by 
\begin{equation}\label{ps}
\psi(x)=
\begin{cases}
\ds |x|^{2-N} &\mbox{ if } m\geq s+\frac{N}{N-2} \mbox{ and } p>q+\frac{N}{N-2},\\[0.2cm]
\ds |x|^{2-N}\log^{\frac{1}{1+s}} \Big(\frac{|x|}{r_0}\Big)  &\mbox{ if } m= s+\frac{N}{N-2}\, , \; p=q+\frac{N}{N-2} \mbox{ and } q>1+s,\\[0.2cm]
\ds |x|^{-\frac{m(N-2)-2}{1+s}} &\mbox{ if }\frac{2}{N-2}<m< s+\frac{N}{N-2}\, , p>\frac{q}{1+s}\Big(m-\frac{2}{N-2}\Big)+\frac{N}{N-2}.
\end{cases}
\end{equation}
and $D,E,F,G>0$ are suitable fixed constants that will be made precise later. For any $(u,v)\in\mathcal{A}$, let $(Tu, Tv)$ be the unique solution of 
\begin{equation}\label{TGM}
\begin{cases}
\displaystyle   -\Delta (Tu)=\frac{u^p}{v^q}+\lambda \rho(x) \,, Tu>0 &\quad\mbox{ in }\R^N\setminus K,\\[0.1in]
\displaystyle   -\Delta (Tv)=u^m (Tv)^{-s}  \,, Tv>0 &\quad\mbox{ in }\R^N\setminus K,\\[0.1in]
\displaystyle \;\;\; \frac{\partial (Tu)}{\partial \nu}=\frac{\partial (Tv)}{\partial \nu}=0 &\quad\mbox{ on }\partial K,\\[0.1in]
\displaystyle \;\;\;  Tu(x), Tv(x)\to 0 &\quad\mbox{ as }|x|\to \infty.
\end{cases}
\end{equation}
Set next
$$
\mathcal{H}:\mathcal{A}\to C(\overline{\R^N\setminus K})\times C(\overline{\R^N\setminus K})\,, \quad \mathcal{H}[u,v]=(Tu, Tv).
$$
\medskip

\noindent{\bf Step 1:} The mapping $\mathcal{H}$ is well defined. 

We show that for each $(u,v)\in \mathcal{A}$, there exist unique  $Tu$ and $Tv$ in the space of twice continuously differentiable  functions $C^2(\overline{\R^N\setminus K})$ that satisfy \eqref{TGM}.  Indeed, let us note first that $(u,v)\in \mathcal{A}$ and \eqref{ro} yield
$$
\frac{u^p}{v^q}+\lambda \rho(x) \simeq |x|^{-p(N-2)}\psi(x)^{-q}+|x|^{-k}\quad\mbox{ in }\R^N\setminus K.
$$
Thus, if condition (i) or (iii) in Theorem \ref{th2} holds, then from \eqref{ps} we find
$$
\frac{u^p}{v^q}+\lambda \rho(x) \simeq |x|^{-\beta}\quad\mbox{ in }\R^N\setminus K,
$$
where
\begin{equation}\label{bb}
\beta=
\begin{cases}
\min\bigg\{ k, p(N-2)-q\frac{m(N-2)-2}{1+s}\bigg\}&\mbox{ if }
\begin{cases}
\frac{2}{N-2}<m<s+\frac{N}{N-2}\, ,\\[0.3cm]
p>\frac{q}{1+s}\Big(m-\frac{2}{N-2}\Big)+\frac{N}{N-2}\,,
\end{cases}\\[0.7cm]
\min\big\{ k, (p-q)(N-2)\big\}&\mbox{ if } 
\begin{cases}
m\geq s+\frac{N}{N-2}\, ,\\[0.3cm]
p>q+\frac{N}{N-2}.
\end{cases}
\end{cases}
\end{equation}
Due to conditions (i) and (iii) in Theorem \ref{th2} and since $k>N$, we have $\beta>N$. Thus, by Lemma \ref{a4} there exists a unique $Tu\in C^2(\overline{\R^N\setminus K})$ which satisfies \eqref{TGM}.

If condition (ii) in Theorem \ref{th2} holds,  then $m = s+\frac{N}{N-2}$, $p=q+\frac{N}{N-2}$ and $q>1+s$.  Thus,
$$
\frac{u^p}{v^q}+\lambda \rho(x) 
\simeq |x|^{-N} \log^{-\frac{q}{1+s}} \Big(\frac{|x|}{r_0}\Big)+ |x|^{-k}\simeq |x|^{-N} \log^{-\frac{q}{1+s}} \Big(\frac{|x|}{r_0}\Big)\quad\mbox{ in }\R^N\setminus K.
$$
We may now apply Lemma \ref{a5} to derive the existence and uniqueness of $Tu\in C^2(\overline{\R^N\setminus K})$ as a unique solution to \eqref{TGM}.

Since $u(x)\simeq |x|^{2-N}$ in $\R^N\setminus K$, we see that $Tv$ defined in \eqref{TGM} satisfies the problem \eqref{inaa1} with $\phi(x)=u^m\simeq |x|^{-m(N-2)}$ and $m(N-2)>2$. By Lemma \ref{a2}, there exists a unique solution $Tv\in C^2(\overline{\R^N\setminus K})$ to \eqref{TGM}. Thus, $\mathcal{H}$ is well defined.

\medskip

\noindent{\bf Step 2:} There exist positive constants $D,E,F,G>0$ and $\lambda^*>0$ such that 
\begin{equation}\label{hal}
\mathcal{H}(\mathcal{A})\subseteq \mathcal{A} \quad\mbox{ for all }0<\lambda\leq \lambda^*.
\end{equation} 

Let $\alpha=m(N-2)>2$.
\begin{lemma}\label{cc}
Let $\psi(x)$ be given by \eqref{ps} and $h(x)=|x|^{-p(N-2)}\psi(x)^{-q}+|x|^{-k}$. Then, there exist two positive constants $C_4>C_3>0$ depending on $N,p,q,m,s$ (but not on $\lambda$) such that:
\begin{enumerate}
\item[{\rm (i)}]  Any function $w\in C^2(\overline{\R^N\setminus K})$ which is a solution of 
\begin{equation}\label{in1}
\begin{cases}
-\Delta w  \geq   |x|^{-\alpha}   w^{-s} \, , w>0 & \quad\mbox{ in }\R^N\setminus K , \\[0.2cm]
\ds \;\;\; \frac{\partial w}{\partial \nu}  =0  & \quad\mbox{ on }\partial K , \\[0.2cm]
\;\;\; w(x)\to 0 &  \quad\mbox{ as }|x|\to \infty,
\end{cases}
\end{equation}
satisfies $w(x)\geq C_3 \psi(x)$ in $\R^N\setminus K$.
\item[{\rm (ii)}] Any function $w\in C^2(\overline{\R^N\setminus K})$ which is a solution of 
\begin{equation}\label{in2}
\begin{cases}
-\Delta w  \leq   |x|^{-\alpha}   w^{-s} \, , w>0 & \quad\mbox{ in }\R^N\setminus K , \\[0.2cm]
\ds \;\;\; \frac{\partial w}{\partial \nu}  =0  & \quad\mbox{ on }\partial K , \\[0.2cm]
\;\;\; w(x)\to 0 &  \quad\mbox{ as }|x|\to \infty,
\end{cases}
\end{equation}
satisfies $w(x)\leq C_4 \psi(x)$ in $\R^N\setminus K$.
\item[{\rm (iii)}] Any function $w\in C^2(\overline{\R^N\setminus K})$ which is a solution of 
\begin{equation}\label{in3}
\begin{cases}
-\Delta w  \geq   |x|^{-k} \, , w>0 & \quad\mbox{ in }\R^N\setminus K , \\[0.2cm]
\ds \;\;\; \frac{\partial w}{\partial \nu}  =0  & \quad\mbox{ on }\partial K , \\[0.2cm]
\;\;\; w(x)\to 0 &  \quad\mbox{ as }|x|\to \infty,
\end{cases}
\end{equation}
satisfies $w(x)\geq C_3 |x|^{2-N}$ in $\R^N\setminus K$.
\item[{\rm (iv)}] Any function $w\in C^2(\overline{\R^N\setminus K})$ which is a solution of 
\begin{equation}\label{in4}
\begin{cases}
-\Delta w  \leq   h(x) \, , w>0 & \quad\mbox{ in }\R^N\setminus K , \\[0.2cm]
\ds \;\;\; \frac{\partial w}{\partial \nu}  =0  & \quad\mbox{ on }\partial K , \\[0.2cm]
\;\;\; w(x)\to 0 &  \quad\mbox{ as }|x|\to \infty,
\end{cases}
\end{equation}
satisfies $w(x)\leq C_4 |x|^{2-N}$ in $\R^N\setminus K$.
\end{enumerate}
\end{lemma}
\begin{proof}
Part (i)-(ii) follow from Lemma \ref{a3} and the comparison principle in Lemma  \ref{com2}. Note also that
$$
\begin{aligned}
h(x)=|x|^{-k}+
\begin{cases}
\ds |x|^{-(N-2)(p-q)} & \mbox{if } 
\begin{cases}
m\geq s+\frac{N}{N-2}\, , \\[0.2cm]
p>q+\frac{N}{N-2}\, ,
\end{cases}\\[0.7cm]
\ds |x|^{-N}\log^{\frac{q}{1+s}} \Big(\frac{|x|}{r_0}\Big)  & \mbox{if } 
\begin{cases}
m= s+\frac{N}{N-2}\, , \\[0.2cm]
p=q+\frac{N}{N-2} \mbox{ and } q>1+s\, ,
\end{cases}\\[0.7cm]
\ds |x|^{-p(N-2)+q\frac{m(N-2)-2}{1+s}} &\mbox{if }
\begin{cases}
\frac{2}{N-2}<m< s+\frac{N}{N-2}\, ,\\[0.3cm]
p>\frac{q}{1+s}\Big(m-\frac{2}{N-2}\Big)+\frac{N}{N-2}.
\end{cases}
\end{cases}
\end{aligned}
$$

If condition (i) or (iii) in Theorem \ref{th2} holds, then $h(x)\simeq |x|^{-\beta}$ in $\R^N\setminus K$, where $\beta>0$ is given by \eqref{bb}. By conditions (i) and (iii) in Theorem \ref{th2} we have $\beta>N$. Now, by Lemma \ref{a4} and the comparison principle in Lemma  \ref{com2} we deduce the statements (iii) and (iv) in our Lemma \ref{cc}.

Finally, if condition (ii) in Theorem \ref{th2} holds, then $h(x)\simeq |x|^{-N}\log^{-\theta}$ in $\R^N\setminus K$, with $\theta=\frac{q}{1+s}>1$.  Part (iii)-(iv) follow now from Lemma \ref{a5} and Lemma \ref{com2}.
\end{proof}
We have introduced the constants $C_2>C_1>0$ in \eqref{ro} and $C_4>C_3>0$ in Lemma \ref{cc}. 
Let next define the new constant $C_5$ by
$$
C_5=C_1^{\frac{m}{1+s}} C_c^{1+\frac{m}{1+s}}>0.
$$
Set
\begin{equation}\label{la1}
\lambda^*=\Bigg( \frac{C_5^q}{(2C_4)^p C_2^{p-1}} \Bigg)^{\frac{1}{(p-1)(1-\sigma)}}>0
\end{equation}
and define
\begin{equation}\label{la2}
\begin{aligned}
D& =C_1C_3\lambda \, , \quad &E=2C_2C_4 \lambda\, , \\[0.2cm]
F& =\ds C_3D^{\frac{m}{1+s}} \, , \quad &G\ds =C_4 E^{\frac{m}{1+s}}.
\end{aligned}
\end{equation}
\medskip

\noindent{\bf Claim:} We have
\begin{equation}\label{la3}
C_4(E^p F^{-q} +\lambda C_2)\leq E \quad\mbox{ for all }0<\lambda\leq \lambda^*.
\end{equation}
Indeed, using \eqref{la2} we see that \eqref{la3} is equivalent to
\begin{equation}\label{la4}
C_4E^p F^{-q}\leq E -\lambda C_2C_4=\lambda C_2 C_4 \Longleftrightarrow E^pF^{-q}\leq C_2\lambda. 
\end{equation}
Using \eqref{la2} we have that \eqref{la4} is equivalent to
$$
(2C_2C_4)^{p} C_3^{-q}(C_1C_3)^{-\frac{mq}{1+s}} \lambda^{p-\frac{mq}{1+s}}\leq C_2\lambda,
$$
that is,
$$
\lambda^{p-1-\frac{mq}{1+s}}\leq \frac{C_5^q}{ (2C_4)^p C_2^{p-1}}.
$$
Since $p-1-\frac{mq}{1+s}=p-1-\sigma(p-1)=(p-1)(1-\sigma)>0$, we see that in light of \eqref{la1}, the above inequality is equivalent to $\lambda\leq \lambda^*$ and concludes the proof of \eqref{la3}.

With the constants $\lambda^*>0$ and $D,E,F,G>0$ defined in \eqref{la1} and \eqref{la2} respectively, we can now proceed to prove \eqref{hal}.

Let $(u, v)\in \mathcal{A}$. From \eqref{TGM} and \eqref{ro} we have
$$
-\Delta (Tu)\geq \lambda \rho(x)\geq C_1\lambda |x|^{-k}\quad\mbox{ in }\R^N\setminus K.
$$
Hence, $w=\frac{Tu}{C_1\lambda}$ satisfies \eqref{in3}. 
By Lemma \ref{cc}(iii) and \eqref{la2}, it follows that
\begin{equation}\label{la5}
w=\frac{Tu}{C_1\lambda}\geq C_3|x|^{2-N} \Longrightarrow Tu\geq  C_1 C_3 \lambda |x|^{2-N}=D|x|^{2-N} \quad\mbox{ in }\R^N\setminus K.
\end{equation}
From \eqref{TGM} and the definition of $\mathcal{A}$ in \eqref{aset}  we find
$$
\begin{aligned}
-\Delta (Tu)&\leq \frac{u^p}{v^q}+C_2\lambda |x|^{-k}\\[0.2cm]
&\leq E^p F^{-q}|x|^{-p(N-2)}\psi(x)^{-q}+C_2\lambda |x|^{-k}\\[0.2cm]
&\leq \big( E^p F^{-q}+C_2\lambda\big) h(x) \quad\mbox{ in }\R^N\setminus K,
\end{aligned}
$$
where $h(x)=|x|^{-p(N-2)}\psi(x)^{-q}+|x|^{-k}$ is defined in Lemma \ref{cc}. Thus, 
$$
w=\frac{Tu}{E^p F^{-q}+C_2\lambda} \quad\mbox{ satisfies \eqref{in4}}. 
$$
By Lemma \ref{cc}(iv), it follows that
$$
w=\frac{Tu}{E^p F^{-q}+C_2\lambda} \leq C_4 |x|^{2-N} \quad\mbox{ in }\R^N\setminus K.
$$
By \eqref{la3} and the above estimate we derive
\begin{equation}\label{la6}
Tu\leq C_4(E^p F^{-q}+C_2\lambda)|x|^{2-N}\leq E|x|^{2-N} \quad\mbox{ in }\R^N\setminus K.
\end{equation}
From \eqref{la5} and \eqref{la6} we now obtain $D|x|^{2-N}\leq Tu\leq E |x|^{2-N}$ in $\R^N\setminus K$.

We next show that $F\psi(x) \leq Tv\leq G\psi(x)$ in $\R^N\setminus K$. First, from $u\geq D|x|^{2-N}$ in $\R^N\setminus K$ and the definition of $Tv$ in \eqref{TGM} we have
$$
-\Delta (Tv)=u^m (Tv)^{-s}\geq D^m |x|^{-\alpha} (Tv)^{-s}\quad\mbox{ in }\R^N\setminus K,
$$
where $\alpha=m(N-2)>2$. A direct calculation  shows that $w=D^{-\frac{m}{1+s}} Tv$ satisfies \eqref{in1} and thus, by Lemma \ref{cc}(i) and \eqref{la2} we deduce
\begin{equation}\label{la7}
w=D^{-\frac{m}{1+s}} Tv\geq C_3 \psi(x)\Longrightarrow Tv\geq C_3 D^{\frac{m}{1+s}}\psi(x)=F\psi(x) \quad\mbox{ in }\R^N\setminus K.
\end{equation}
Finally, from $u\leq E|x|^{2-N}$ in $\R^N\setminus K$ and the definition of $Tv$ in \eqref{TGM} we obtain
$$
-\Delta (Tv)=u^m (Tv)^{-s}\leq E^m |x|^{-\alpha} (Tv)^{-s}\quad\mbox{ in }\R^N\setminus K.
$$
This implies that $w=E^{-\frac{m}{1+s}} Tv$ satisfies \eqref{in2} and  by Lemma \ref{cc}(ii) and \eqref{la2} we find
\begin{equation}\label{la8}
w=E^{-\frac{m}{1+s}} Tv\leq C_4 \psi(x)\Longrightarrow Tv\leq C_4 E^{\frac{m}{1+s}}\psi(x)=G\psi(x) \quad\mbox{ in }\R^N\setminus K.
\end{equation}
From \eqref{la7}-\eqref{la8} we deduce $ F\psi(x)\leq Tv\leq G\psi(x)$ in $\R^N\setminus K$, and this concludes our proof in Step 2.

\medskip

\noindent{\bf Step 3:}  Existence of a solution  $(u,v)\in \mathcal{A}$ to \eqref{GM0}.

We have seen that $\mathcal{H}(\mathcal{A}) \subset \mathcal{A}$. Due to the lack of a compactness argument, we cannot apply directly the Schauder Fixed Point Theorem on $\mathcal{A}$. To simplify our further notations, let us assume $K\subset B_1$.
For any $n\geq 1$ we introduce
$$
\mathcal{A}_n=\left\{(u,v)\in C^{0,\gamma}(\overline{B_n\setminus K})\times C^{0,\gamma}(\overline{B_n\setminus K}): 
\begin{aligned}
D|x|^{2-N} & \leq u\leq  E |x|^{2-N}\\
F\psi(x)  & \leq v\leq  G \psi(x)
\end{aligned} \; \mbox{ in } \overline{B_n \setminus K}\right\},
$$
where $E>D>0$ and $G>F>0$ are the constants that satisfy \eqref{la2}.
\smallskip

Let $(u,v)\in \mathcal{A}_n$. We fix $(\widetilde u, \widetilde v)\in \mathcal{A}$, a $C^{0,\gamma}_{loc}$-extension of $(u, v)$ to the whole $\overline{\R^N\setminus K}$.  Let $(U, V)=(T\widetilde u, T\widetilde v)$ be the solution of  \eqref{TGM} with the input data $(\widetilde u, \widetilde v)\in \mathcal{A}$.  We define
$$
\mathcal{H}_n:\mathcal{A}_n\to C(\overline{B_n\setminus K})\times C(\overline{B_n\setminus K})\quad\mbox{ by }\quad  
\mathcal{H}_n[u, v]:=\Big(U\!\mid_{\overline{B_n\setminus K}}, V\!\mid_{\overline{B_n\setminus K}}\Big).
$$
One can see that $\mathcal{H}_n$ is continuous  in the sense that if $\{(u_k, v_k)\}_{k\geq 1}$ converges to $(u, v)$ in $C^{0,\gamma}(\overline{B_n\setminus K})\times C^{0,\gamma}(\overline{B_n\setminus K})$ as $k\to \infty$, then, by the continuous dependence on data for semilinear PDEs, one has $T\widetilde u_k\to T\widetilde u$ and $T\widetilde v_k\to T\widetilde v$ in $C^{0,\gamma}(\overline{B_n\setminus K})$ as $k\to \infty$, so $\mathcal{H}_n$ is continuous.

Since $\mathcal{H}(\mathcal{A})\subset \mathcal{A}$ we have $\mathcal{H}_n(\mathcal{A}_n)\subset \mathcal{A}_n$ for all $n\geq 1$. Furthermore, by standard elliptic regularity, $\mathcal{H}_n$ maps $\mathcal{A}_n$ into $C^1(\overline{B_n\setminus K})\times C^1(\overline{B_n\setminus K})$ which is compactly embedded into $C^{0,\gamma}(\overline{B_n\setminus K})\times C^{0,\gamma}(\overline{B_n\setminus K})$ and thus into $\mathcal{A}_n$. This shows that  
$$
\mathcal{H}_n: \mathcal{A}_n\to \mathcal{A}_n\quad\mbox{ is compact.}
$$
Hence,  $\mathcal{H}_n$ fulfils all conditions in the Schauder Fixed Point Theorem  and thus, there exists $(u_n, v_n)\in \mathcal{A}_n$ a fixed point of $\mathcal{H}_n$. This yields 
\begin{equation}\label{TGMn}
\begin{cases}
\displaystyle   -\Delta u_n=\frac{u_n^p}{v_n^q}+\lambda \rho(x) \,, u_n>0 &\quad\mbox{ in }B_n\setminus K,\\[0.15in]
\displaystyle   -\Delta v_n=\frac{u_n^m}{v_n^s}   \,, v_n>0 &\quad\mbox{ in } B_n\setminus K,\\[0.15in]
\displaystyle \;\;\; \frac{\partial u_n}{\partial \nu}=\frac{\partial v_n}{\partial \nu}=0 &\quad\mbox{ on }\partial K.
\end{cases}
\end{equation}
From \eqref{TGMn} we have 
$$
-\Delta u_n=f_n(x) \quad\mbox{ and }\quad -\Delta v_n=h_n(x)\quad\mbox{ in }B_n\setminus K,
$$
where, by the definition of $\mathcal{A}_n$, for all $x\in B_n\setminus K$ we have
\begin{equation}\label{en}
\begin{cases}
\ds \; f_n(x)=\frac{u_n^p}{v_n^q}+\lambda \rho(x)\leq (E^pF^{-q}+C_2\lambda )(|x|^{-p(N-2)}\psi(x)^{-q}+|x|^{-k})\\[0.1in]
\ds \; h_n(x)=\frac{u_n^m}{v_n^s}\leq E^m F^{-s} |x|^{-m(N-2)}\psi(x)^{-s}.
\end{cases}
\end{equation}
The above coefficients $E,F,C_2>0$ are independent of $n$ and $f_n,h_n$ are bounded in $L^r(B_n\setminus K)$ for $r>1$ large. 
By standard elliptic estimates, we now obtain that $\{(u_n, v_n)\}_{n\geq 1}$ is bounded in $W^{2,r}(B_n\setminus K)\times W^{2,r}(B_n\setminus K)$ for all $r>1$ large. By the Sobolev embeddings, it follows that $\{(u_n, v_n)\}_{n\geq 1}$ is bounded in 
$C^{1,\gamma}(\overline{B_1\setminus K})\times C^{1,\gamma}(\overline{B_1\setminus K})$. Since the embedding $C^{1,\gamma}(\overline{B_1\setminus K}) \hookrightarrow C^{1}(\overline{B_1\setminus K})$ is compact, it follows that $\{(u_n, v_n)\}_{n\geq 1}$ has a convergent subsequence $\{(u^1_n, v^1_n)\}_{n\geq 1}$ in $C^{1}(\overline{B_1\setminus K})\times C^1(\overline{B_1\setminus K})$. Next, we proceed similarly with the sequence $\{(u^1_n, v^1_n)\}_{n\geq 2}$ which, by elliptic estimates, it is bounded in $W^{2,r}(\overline{B_2\setminus K})\times W^{2,r}(\overline{B_2\setminus K})$ for all $r>1$ large and thus, has a convergent subsequence $\{(u^2_n, v^2_n)\}_{n\geq 2}$ in $C^{1}(\overline{B_2\setminus K})\times C^{1}(\overline{B_2\setminus K})$. 

To summarise, for all $m\geq 1$ we may construct a sequence 
$$
\{(u^m_n, v^m_n)\}_{n\geq m}\subset \{(u^{m-1}_n, v^{m-1}_n)\}_{n\geq m}
$$ which is convergent in $C^{1}(\overline{B_m\setminus K})\times C^{1}(\overline{B_m\setminus K})$. Now, the diagonal sequence $\{(u^n_n, v^n_n)\}_{n\geq 1}$ will be convergent in $C^{1}_{loc}(\overline{\R^N\setminus K})\times C^{1}_{loc}(\overline{\R^N\setminus K})$. Its limit $(u,v)$ is a $C^1$-weak solution of \eqref{GM0}. Note that $(u_n, v_n)\in \mathcal{A}_n$ for all $n\geq 1$ implies $(u,v)\in \mathcal{A}$. In particular, $u$ satisfies \eqref{beh1} and $u(x), v(x)\to 0$ as $|x|\to \infty$. From \eqref{en} we have
$$
\begin{aligned}
\frac{u^p}{v^q}+\lambda \rho(x)& \leq (E^pF^{-q}+C_2\lambda )(|x|^{-p(N-2)}\psi(x)^{-q}+|x|^{-k})\\[0.1in]
\frac{u^m}{v^s} & \leq E^m F^{-s} |x|^{-m(N-2)}\psi(x)^{-s}
\end{aligned}
\quad\mbox{ in }\R^N\setminus K.
$$
By regularity theory (see \cite{ADN1,ADN2}), it follows that $(u, v)\in C^{2,\gamma}_{loc}(\overline{\R^N\setminus K})\times C^{2,\gamma}_{loc}(\overline{\R^N\setminus K})$. This concludes our proof.

\qed

\begin{remark}\label{remm}
{\rm From the asymptotic behaviour of solutions $u$ and $v$ which is given in \eqref{beh1} and \eqref{beh2} respectively, we can easily see that if condition (ii) or (iii) in Theorem \ref{th2} holds, then 
$$
v(x)>>u(x)\quad\mbox{ for } |x| \mbox{ large}.
$$

If condition (i) in Theorem \ref{th2} holds, we constructed the solution $(u,v)$ with the property $u(x)\simeq v(x)\simeq |x|^{2-N}$ as $|x|\to \infty$. Let us notice that $(u,v)\in \mathcal{A}$ (see \eqref{aset}) where the coefficients $D,E,F,G$ defined in \eqref{la2} satisfy
$$
D\simeq E\simeq \lambda\quad\mbox{ and }\quad F\simeq G\simeq \lambda^{\frac{m}{s+1}}\quad\mbox{ as }\lambda\to 0.
$$
Since condition (i) in Theorem \ref{th2} implies $m\geq   s+\frac{N}{N-2}>s+1$, it follows that 
$$
D\simeq \lambda>>\lambda^{\frac{m}{s+1}}\simeq G \quad\mbox{ as }\lambda\to 0,
$$
and thus, when $\lambda>0$ is small one has $u(x)>v(x)$ as $|x|$ is large.
}

\end{remark}

\section{Proof of Theorem \ref{th3}}

Assume first that $(u,v)$ is a positive solution to \eqref{GM0} and $u$ satisfies \eqref{beh3}. Then $v$ satisfies \eqref{inaa1} in Lemma \ref{a3} where $\phi(x)=u(x)^m\simeq |x|^{-am}$ in $\R^N\setminus K$. By Lemma \ref{a3} it follows that $am>2$ and that $v$ satisfies the asymptotic behaviour \eqref{beh4}. Let 
\begin{equation}\label{fi1}
\varphi(x)= \begin{cases}
\ds |x|^{2-N} & \quad \mbox{ if } m>\frac{N+s(N-2)}{a} \\[0.2cm]
\ds |x|^{2-N}\log^{\frac{1}{1+s}} \Big(\frac{|x|}{r_0}\Big)  &\quad \mbox{ if } m=\frac{N+s(N-2)}{a} \\[0.2cm]
\ds |x|^{-\frac{ma-2}{1+s}} &\quad \mbox{ if } \frac{2}{a}<m< \frac{N+s(N-2)}{a}
\end{cases}
\quad\mbox{ for all } x\in \R^N\setminus K.
\end{equation}
Now, \eqref{beh4} yields $v(x)\simeq \varphi(x)$ in $\R^N\setminus K$.
Using this fact in the first equation of \eqref{GM0} we derive that $u$ satisfies the problem \eqref{corw1} where
\begin{equation}\label{beh5}
h(x)\simeq |x|^{-ap}\varphi(x)^{-q}+|x|^{-a-2}\quad\mbox{ in }\R^N\setminus K.
\end{equation}
\noindent Assume first $m> \frac{N+s(N-2)}{a}$. Then by \eqref{fi1} and \eqref{beh5} we have $h(x)\simeq |x|^{-\beta}$ in $\R^N\setminus K$, where
$$
\beta=\min\Big\{ap-q(N-2), a+2  \Big\}.
$$
Since $u$ satisfies \eqref{beh3}, by Lemma \ref{a4} we deduce $\beta=a+2<N$. This yields  $ap-q(N-2)\geq a+2$, thus $p\geq q\frac{N-2}{a}+1+\frac{2}{a}$. Hence,  condition (i) in the statement of Theorem \ref{th3} holds.

\noindent Assume now $\frac{2}{a}<m<\frac{N+s(N-2)}{a}$. As before, we have $h(x)\simeq |x|^{-\beta}$ in $\R^N\setminus K$, where
$$
\beta=\min\Big\{ap-\frac{q}{1+s}\big(am-2\big), a+2  \Big\}.
$$
Further, \eqref{beh3} and Lemma \ref{a4} yield  $\beta=a+2<N$, so $p\geq \frac{q}{1+s}\Big(m-\frac{2}{a}\Big)+1+\frac{2}{a}$.
This shows that condition (ii) in the statement of Theorem \ref{th3} holds.

\noindent Finally, in the case $m=\frac{N+s(N-2)}{a}$ we have from \eqref{fi1} and \eqref{beh5}  that
$$
h(x)\simeq |x|^{-ap+q(N-2)}\log^{-\frac{q}{1+s}} \Big(\frac{|x|}{r_0}\Big)+|x|^{-a-2}\quad\mbox{ in }\R^N\setminus K.
$$
If $ap-q(N-2)\geq a+2$ then $h(x)\simeq |x|^{-a-2}$ in $\R^N\setminus K$. By Lemma \ref{a4} we deduce $u(x)\simeq |x|^{-a}$.

If $ap-q(N-2)< a+2$, then $h(x)\geq C_\varepsilon |x|^{-a-2+\varepsilon}$ in $\R^N\setminus K$, for some $\varepsilon>0$ small and $C_\varepsilon>0$. Similar to Lemma \ref{a4} we derive $u(x)\geq c|x|^{-a+\varepsilon}$ in $\R^N\setminus K$, which contradicts \eqref{beh3}. Hence, $ap-q(N-2)\geq a+2$, which yields $p\geq q\frac{N-2}{a}+1+\frac{2}{a}$ and thus condition (i) in Theorem \ref{th3} holds.

Assume now that one of conditions (i) and (ii) hold. The construction of a positive solution $(u,v)$ to \eqref{GM0} that satisfies \eqref{beh3} follows the same approach as in Theorem \ref{th2}. We define 
$$
\mathcal{A}=\Bigg\{(u, v)\in C^{0,\gamma}_{loc}(\overline{\R^N\setminus K})\times C^{0,\gamma}_{loc}(\overline{\R^N\setminus K}): \;
\begin{aligned}
D|x|^{-a}&\leq u(x)\leq E |x|^{-a}\\
F\varphi (x)&\leq v(x)\leq G \varphi (x)
\end{aligned} \; \mbox{ in }\R^N\setminus K\Bigg\},
$$
where $\varphi(x) $ is given by \eqref{fi1} and $D,E,F,G>0$ are suitable fixed constants. From now on we follow step by step the proof of Theorem \ref{th2}.
\qed

\section{Further extensions}

The approach we developed in this work can be  extended to more other systems with various sign of exponents.
Consider for instance the systems
\begin{equation}\label{GMe1}
\begin{cases}
\displaystyle   -\Delta u=\frac{1}{u^p v^q}+\lambda \rho(x) \,, u>0 &\quad\mbox{ in }\R^N\setminus K,\\[0.1in]
\displaystyle   -\Delta v=\frac{u^m}{v^s}  \,, v>0 &\quad\mbox{ in }\R^N\setminus K,\\[0.1in]
\displaystyle \;\;\; \frac{\partial u}{\partial \nu}=\frac{\partial v}{\partial \nu}=0 &\quad\mbox{ on }\partial K,\\[0.1in]
\displaystyle \;\;\;  u(x), v(x)\to 0 &\quad\mbox{ as }|x|\to \infty,
\end{cases}
\end{equation}
\begin{equation}\label{GMe2}
\begin{cases}
\displaystyle   -\Delta u=\frac{1}{u^p v^q}+\lambda \rho(x) \,, u>0 &\quad\mbox{ in }\R^N\setminus K,\\[0.1in]
\displaystyle   -\Delta v=\frac{1}{u^m v^s}  \,, v>0 &\quad\mbox{ in }\R^N\setminus K,\\[0.1in]
\displaystyle \;\;\; \frac{\partial u}{\partial \nu}=\frac{\partial v}{\partial \nu}=0 &\quad\mbox{ on }\partial K,\\[0.1in]
\displaystyle \;\;\;  u(x), v(x)\to 0 &\quad\mbox{ as }|x|\to \infty,
\end{cases}
\end{equation}
and
\begin{equation}\label{GMe3}
\begin{cases}
\displaystyle   -\Delta u=\frac{v^q}{u^p}+\lambda \rho(x) \,, u>0 &\quad\mbox{ in }\R^N\setminus K,\\[0.1in]
\displaystyle   -\Delta v=\frac{u^m}{v^s}  \,, v>0 &\quad\mbox{ in }\R^N\setminus K,\\[0.1in]
\displaystyle \;\;\; \frac{\partial u}{\partial \nu}=\frac{\partial v}{\partial \nu}=0 &\quad\mbox{ on }\partial K,\\[0.1in]
\displaystyle \;\;\;  u(x), v(x)\to 0 &\quad\mbox{ as }|x|\to \infty,
\end{cases}
\end{equation}
where $\rho\in C^{0,\gamma}_{loc}(\overline{\R^N\setminus K})$, $\gamma\in (0,1)$, is a nonnegative function and $p,q,m,s, \lambda>0$. 

Using an approach from dynamical systems, positive radial solutions of the following system 
$$
\begin{cases}
-\Delta u=|x|^{\alpha} u^a v^b \\[0.1cm]
-\Delta v=|x|^{\beta} u^c v^d
\end{cases}
\quad\mbox{ in }\Omega,
$$
were studied in \cite{BG10} for $\Omega=\R^N$ and in \cite{BR96} for $\Omega=\R^N\setminus\{0\}$ where $a,b,c,d, \alpha, \beta\in \R$ (see also \cite{BP01} for further extensions to quasilinear elliptic operators). Our systems \eqref{GMe1}-\eqref{GMe3} are posed in an exterior domain $\R^N\setminus K$ and complemented by Neumann boundary condition on $\partial K$. Thus, \eqref{GMe1}-\eqref{GMe3} are not suited for radial symmetric solutions due to the general boundary of $K$.

Using Lemma \ref{a1} we obtain the following nonexistence result.

\begin{theorem}{\rm (Nonexistence)}\label{the1}

Assume $N\geq 2$ and $p,q,m,s>0$, $\lambda\geq 0$. 
\begin{enumerate}
\item[{\rm (i)}] The systems \eqref{GMe1} and \eqref{GMe2} have no positive solutions.
\item[{\rm (ii)}] If one of the following conditions hold
\begin{enumerate}
\item[{\rm (ii1)}] $N=2$;
\item[{\rm (ii2)}] $N\geq 3$ and $\ds \min\{ q,m\}\leq \frac{2}{N-2}$;
\end{enumerate} 
 then \eqref{GMe3} has no positive solutions.
 \end{enumerate}
\end{theorem}
\begin{proof}
(i) Suppose $(u,v)$ is a positive solution of either \eqref{GMe1} or \eqref{GMe2}. Since $v(x)\to 0$ as $|x|\to\infty$, there exists $M>0$ such that $v(x)\leq M$ for all $x\in \R^N\setminus K$. Then, from \eqref{GMe1} and \eqref{GMe2} we find $-\Delta u\geq M^{-q} u^{-p}$ in $\R^N\setminus K$.  Now by Lemma \ref{a1}(i)-(ii)  with $A(x)\equiv M^{-q}$ it follows that \eqref{GMe1} and \eqref{GMe2} have no positive solutions.

(ii) Suppose $(u,v)$ is a positive solution of \eqref{GMe3}.  Since $u,v$ are superharmonic, by Lemma \ref{lest}  it follows that
$$
\begin{cases}
u, v\geq c &\quad\mbox{ in }\R^N\setminus K,  \mbox{ if }N=2,\\
u, v\geq c|x|^{2-N} & \quad\mbox{ in }\R^N\setminus K,  \mbox{ if }N\geq 3,
\end{cases}
$$
where $c>0$ is a constant.

(ii1) If $N=2$,  we derive from Lemma \ref{a1}(ii) that there are no positive solutions of \eqref{GMe3}.

 (ii2) If $N\geq 3$ we deduce that $u$ and $v$ satisfy 
$$
\begin{cases}
-\Delta u\geq c^q |x|^{-q(N-2)} u^{-p} &\quad \mbox{ in }\R^N\setminus K,\\[0.2cm]
-\Delta v\geq c^m |x|^{-m(N-2)} v^{-s} &\quad \mbox{ in }\R^N\setminus K.
\end{cases}
$$
Now if either $q\leq \frac{2}{N-2}$ or $m\leq \frac{2}{N-2}$, Lemma \ref{a1}(i) yields a contradiction.
\end{proof}

We next discuss the existence of positive solutions to \eqref{GMe3}. We present here an optimal result related to the existence of positive solutions $(u,v)$ to \eqref{GMe3} with minimal decay at infinity, that is,
\begin{equation}\label{uvmin}
u(x)\simeq v(x)\simeq |x|^{2-N} \quad\mbox{ in }\R^N\setminus K.
\end{equation} 
The interested reader may explore other features of \eqref{GMe3} such as existence of solutions with different asymptotic behaviour at infinity, as we did in Theorem \ref{th3}. Our last result in this section is stated and proved below.
 
\begin{theorem}{\rm (Existence of solutions to \eqref{GMe3} with  minimal decay)}\label{the2}

Assume $N\geq 3$,  $p,q,m,s>0$ and $\rho(x)$ satisfies \eqref{ro}.  Then, for some $\lambda>0$ the system  \eqref{GMe3} has a positive solution $(u,v)$ which satisfies \eqref{uvmin} if and only if 
\begin{equation}\label{copt}
k>N,\quad \; q>p+\frac{N}{N-2}\quad\mbox{ and }\quad m>s+\frac{N}{N-2}.
\end{equation}
\end{theorem}
\begin{proof} Let us first assume that $(u,v)$ is a positive solution of \eqref{GMe3} which satisfies \eqref{uvmin}. Then, by Theorem \ref{the1} we have $q(N-2)>2$ and $m(N-2)>2$. Also, 
$$
-\Delta u\geq \lambda \rho(x)\geq \lambda C_2|x|^{-k}\quad\mbox{ in }\R^N\setminus K.
$$
By \eqref{uvmin} and the asymptotic \eqref{corw2} in Lemma \ref{a4} it follows that $k>N$. 
Now, the first equation of \eqref{GMe3} together with \eqref{uvmin} yield 
$$
-\Delta u\geq \phi(x) u^{-p}\quad\mbox{  in } \R^N\setminus K,
$$ 
where $\phi(x)=v(x)^q\simeq |x|^{-q(N-2)}$ in $\R^N\setminus K$.  If $q(N-2)\leq N+p(N-2)$, then by the asymptotic behaviour \eqref{lueq1} in Lemma \ref{a3} together with the comparison principle in Lemma \ref{com2} we find
$$
v(x)\geq c
\begin{cases}
|x|^{-\frac{q(N-2)-2}{1+p}} &\quad\mbox{ if } 2<q(N-2)<N+p(N-2)\\[0.1in]
|x|^{2-N}\log^{\frac{1}{1+p}} &\quad\mbox{ if } q(N-2)=N+p(N-2)\\[0.1in]
\end{cases}
\quad\mbox{ in }\R^N\setminus K.
$$
This contradicts the asymptotic \eqref{uvmin}. Hence, $q(N-2)>N+p(N-2)$ which yields  $q>p+\frac{N}{N-2}$. Similarly, $v$ satisfies 
$$
-\Delta v=\psi(x) v^{-s}\quad\mbox{ in } \R^N\setminus K,
$$ 
where $\psi(x)=u(x)^m\simeq |x|^{-m(N-2)}$ in $\R^N\setminus K$. With a similar argument we deduce $m>s+\frac{N}{N-2}$.

\medskip

Conversely, assume now that $p,q,m,s,k$ satisfy \eqref{copt}. The following result is the counterpart of Lemma \ref{cc} which suits our system \eqref{GMe3}.

\begin{lemma}\label{cc1}
Assume \eqref{copt} holds. Then, there exist two positive constants $C_6>C_5>0$ depending on $N,p,q,m,s$ (but not on $\lambda$) such that:
\begin{enumerate}
\item[{\rm (i)}]  Any function $w\in C^2(\overline{\R^N\setminus K})$ which is a solution of 
$$
\begin{cases}
-\Delta w  \geq   |x|^{-m(N-2)}   w^{-s} \, , w>0 & \quad\mbox{ in }\R^N\setminus K , \\[0.2cm]
\ds \;\;\; \frac{\partial w}{\partial \nu}  =0  & \quad\mbox{ on }\partial K , \\[0.2cm]
\;\;\; w(x)\to 0 &  \quad\mbox{ as }|x|\to \infty,
\end{cases}
$$
satisfies $w(x)\geq C_5 |x|^{2-N} $ in $\R^N\setminus K$.
\item[{\rm (ii)}] Any function $w\in C^2(\overline{\R^N\setminus K})$ which is a solution of 
$$
\begin{cases}
-\Delta w  \leq   |x|^{-m(N-2)}   w^{-s} \, , w>0 & \quad\mbox{ in }\R^N\setminus K , \\[0.2cm]
\ds \;\;\; \frac{\partial w}{\partial \nu}  =0  & \quad\mbox{ on }\partial K , \\[0.2cm]
\;\;\; w(x)\to 0 &  \quad\mbox{ as }|x|\to \infty,
\end{cases}
$$
satisfies $w(x)\leq C_6 |x|^{2-N}$ in $\R^N\setminus K$.
\item[{\rm (iii)}] Any function $w\in C^2(\overline{\R^N\setminus K})$ which is a solution of 
$$
\begin{cases}
-\Delta w  \geq   |x|^{-k} \, , w>0 & \quad\mbox{ in }\R^N\setminus K , \\[0.2cm]
\ds \;\;\; \frac{\partial w}{\partial \nu}  =0  & \quad\mbox{ on }\partial K , \\[0.2cm]
\;\;\; w(x)\to 0 &  \quad\mbox{ as }|x|\to \infty,
\end{cases}
$$
satisfies $w(x)\geq C_3 |x|^{2-N}$ in $\R^N\setminus K$.
\item[{\rm (iv)}] Any function $w\in C^2(\overline{\R^N\setminus K})$ which is a solution of 
$$
\begin{cases}
-\Delta w  \leq  |x|^{-(q-p)(N-2)} +|x|^{-k} \, , w>0 & \quad\mbox{ in }\R^N\setminus K , \\[0.2cm]
\ds \;\;\; \frac{\partial w}{\partial \nu}  =0  & \quad\mbox{ on }\partial K , \\[0.2cm]
\;\;\; w(x)\to 0 &  \quad\mbox{ as }|x|\to \infty,
\end{cases}
$$
satisfies $w(x)\leq C_6 |x|^{2-N}$ in $\R^N\setminus K$.
\end{enumerate}
\end{lemma}
The proof of Lemma \ref{cc1} is similar to that of Lemma \ref{cc} and will be omitted here.

Let $C_2>C_1>0$ be the constants from \eqref{ro} and $C_6>C_5>0$ be the constants introduced in Lemma \ref{cc1}. Define
\begin{equation}\label{lac}
D=C_1C_5 \lambda \, ,\quad E= 2C_2C_6 \lambda   \, ,\quad F=D^{\frac{m}{1+s}} C_5 \, ,\quad G=E^{\frac{m}{1+s}} C_6.
\end{equation}
Then $D<E$ and $F<G$. Define next 
$$
\mathcal{A}=\Bigg\{(u, v)\in C^{0,\gamma}_{loc}(\overline{\R^N\setminus K})\times C^{0,\gamma}_{loc}(\overline{\R^N\setminus K}): \;
\begin{aligned}
D|x|^{2-N}&\leq u(x)\leq E |x|^{2-N}\\
F |x|^{2-N} &\leq v(x)\leq G |x|^{2-N}
\end{aligned} \;\mbox{ in }\R^N\setminus K\Bigg\}.
$$
For any $(u,v)\in\mathcal{A}$, let $(Tu, Tv)$ be the unique solution of 
$$
\begin{cases}
\displaystyle   -\Delta (Tu)=\frac{v^q}{u^p}+\lambda \rho(x) \,, Tu>0 &\quad\mbox{ in }\R^N\setminus K,\\[0.1in]
\displaystyle   -\Delta (Tv)=u^m (Tv)^{-s}  \,, Tv>0 &\quad\mbox{ in }\R^N\setminus K,\\[0.1in]
\displaystyle \;\;\; \frac{\partial (Tu)}{\partial \nu}=\frac{\partial (Tv)}{\partial \nu}=0 &\quad\mbox{ on }\partial K,\\[0.1in]
\displaystyle \;\;\;  Tu(x), Tv(x)\to 0 &\quad\mbox{ as }|x|\to \infty.
\end{cases}
$$
We further define
$$
\mathcal{H}:\mathcal{A}\to C(\overline{\R^N\setminus K})\times C(\overline{\R^N\setminus K})\,, \quad \mathcal{H}[u,v]=(Tu, Tv).
$$
Clearly, any fixed point of $\mathcal{H}$ is a solution to \eqref{GMe3}. Let 
$$
\lambda^{**}:=\left[\frac{C_2(C_1C_5)^p}{(2C_2C_6)^{\frac{mq}{1+s}}C_6^q} \right]^{\frac{1}{\frac{mq}{1+s}-(p+1)}}.
$$
From the definition of $D,E,F,G$ in \eqref{lac} we can check that for all $0<\lambda\leq \lambda^{**}$ we have
$$
C_6 (G^qD^{-p}+C_2\lambda)\leq E,
$$
which is the counterpart of \eqref{la3}. From now on, the approach follows the same steps as in the proof of Theorem \ref{th2} to deduce $\mathcal{H}(\mathcal{A})\subset \mathcal{A}$ and to obtain a solution to \eqref{GMe3}.
\end{proof}

\section*{Data availability statement}

No new data were created or analysed in this study.

\section*{Acknowledgements}

This publication has emanated from research conducted with the financial support of Science Foundation Ireland under Grant number 18/CRT/6049.

\section*{ORCID iDs}

\noindent Marius Ghergu: https://orcid.org/0000-0001-9104-5295

\noindent Jack McNichholl: https://orcid.org/0009-0002-3922-6130

\end{document}